\documentclass[reqno]{amsart}
\usepackage{amssymb}
\usepackage{amscd}
\usepackage{amsmath}
\usepackage{enumerate}
\usepackage{graphicx}
\usepackage[all]{xy} 

\usepackage{color}

\theoremstyle{plain} 
\newtheorem{theo}{Theorem}[section] 
\newtheorem{lem}[theo]{Lemma}
\newtheorem{cor}[theo]{Corollary}
\newtheorem{prop}[theo]{Proposition}
\newtheorem{notation}[theo]{Notation}
\newtheorem{rem}[theo]{Remark}

\theoremstyle{definition} 
\newtheorem{defn}[theo]{Definition} 
\newtheorem{conj}{Conjecture}

\theoremstyle{remark} 

\newtheorem{ex}{Example}

\newtheorem{fact}{Fact}

\usepackage{hyperref}
 
\usepackage[all]{xy}

\newcommand{\CC}{{\mathbb{C}}}

\newcommand{\NN}{{\mathbb{N}}}

\newcommand{\QQ}{{\mathbb{Q}}}

\newcommand{\ZZ}{{\mathbb{Z}}}

\newcommand{\calG}{{\mathcal{G}}}
\newcommand{\calH}{{\mathcal{H}}}

\newcommand{\calO}{{\mathcal{O}}}

\newcommand{\calX}{{\mathcal{X}}}

\newcommand{\comp}{{\circ}}
\newcommand{\cont}{{\subset}}
\newcommand{\contneq}{{\subsetneq}}
\newcommand{\ncontneq}{\nsubseteq}

\newcommand{\ord}{{\mathrm{ord}}}

\begin{document}
\title[Generalized Nash problem and adjacencies]{On the generalized Nash problem 
for smooth germs  and adjacencies of curve singularities}
\author{Javier Fern\'andez de Bobadilla}
\address{Javier Fern\'andez de Bobadilla:  
(1) IKERBASQUE, Basque Foundation for Science, Maria Diaz de Haro 3, 48013, 
    Bilbao, Bizkaia, Spain
(2) BCAM  Basque Center for Applied Mathematics, Mazarredo 14, E48009 Bilbao, 
Basque Country, Spain } 
\email{javierbobadilla73@gmail.com}
\author{Mar\'ia Pe Pereira}
\address{Mar\'ia Pe Pereira:  Instituto de Ciencias Matem\'aticas (ICMAT),  Nicol\'as Cabrera, 13-15, 
Campus de Cantoblanco, UAM, 28049
Madrid, Spain.}
\email{maria.pe@icmat.es}
\author{Patrick Popescu-Pampu}
\address{Patrick Popescu-Pampu: Universit{\'e} Lille 1, UFR de Maths., B\^atiment M2, 
     Cit\'e Scientifique, 59655, Villeneuve d'Ascq Cedex, France.}
\email{patrick.popescu@math.univ-lille1.fr}
\date{15-09-2017}

\begin{abstract}
       In this paper we explore the generalized Nash problem for arcs on a germ 
        of smooth surface: given two prime divisors above its special point, to determine 
        whether the arc space of one of them is included in the arc space of the other one. 
        We prove that this problem is combinatorial and we explore its relation with 
        several notions of adjacency of plane curve singularities.
\end{abstract}

\maketitle

\begin{center}
    {\bf This paper appeared in Advances in Mathematics 320 (2017) 1269-1317.
       http://dx.doi.org/10.1016/j.aim.2017.09.035}
\end{center}

\tableofcontents

\section{Introduction}

The original motivation, and one of the main purposes of this paper, is to study 
the \emph{generalized Nash problem} for smooth surface germs.
However, working on this problem we have obtained results in the study 
of adjacencies of plane curve singularities and 
in the study of divisorial valuations in the plane, as we will explain below.

Let $(X,O)$ be a germ of a normal surface singularity. A \emph{model} 
of $X$ is a proper birational 
map $\pi:  S \to X$ for which $S$ is smooth. A \emph{prime divisor over $O$} 
is an irreducible divisor $E$ in a model $S$ such that  $\pi(E)=O$. 
If $h$ is a function in the local ring $\mathcal{O}$ of germs of holomorphic 
functions on $(X,O)$, consider its order of vanishing $\nu_E(h)$ along $E$. 
The function $\nu_E : \mathcal{O} \to \NN \cup \{\infty\}$ is a discrete 
 valuation of rank one.  We call it the \emph{valuation associated with $E$}. 
Two prime divisors over $O$ are \emph{equivalent} if they induce the same valuation. 
A \emph{divisor over} $O$ is a formal finite combination 
$\sum_ia_iE_i$ for $a_i\in \NN$ and $E_i$ prime over $O$. 
We define $\nu_E:=\sum_ia_i\nu_{E_i}$, which is a 
function $\nu_E:\calO\to\ZZ$, but in general no longer a valuation. 

The \emph{maximal divisorial set} $\overline{N}_E$ associated with a prime divisor $E$ 
over $O$ is the Zariski closure in the space of arcs of $(X,O)$ 
of the set of arcs whose liftings to $S$ hit $E$.
It is well known that this definition depends only on the equivalence class of $E$. 
A \emph{Nash adjacency} is an inclusion $\overline{N}_F\cont \overline{N}_E$.  

Given a normal surface singularity $(X,O)$ and given two different prime divisors $E, F$ 
 appearing on its minimal resolution, Nash conjectured in \cite{Nash}
that {\em none of the spaces $\overline{N}_E, \overline{N}_F$ is included in the other one}. 
This conjecture, which was proved in \cite{Nash_surfaces}, can be 
generalized in the form of the following problem which, as far as we know, 
was stated first by Ishii \cite[3.10]{I_max} 
(here $X$ may denote an irreducible germ of arbitrary dimension): 

\medskip
\emph{Let $E$ and $F$ be prime divisors over $O\in X$. Characterize when 
    $\overline{N}_F \subset \overline{N}_E$.}
\medskip

This problem is wide open even for \emph{smooth} germs of surfaces $(X,O)$. 
This is the reason why we concentrate here on this case.  We will assume 
without loss of generality that our germ is $(\CC^2,0)$.

The \emph{combinatorial type} of a pair of divisors is the combinatorial type 
of the minimal sequence of blowing ups needed for 
making both of them appear. One of the main results in this paper is 
(see Theorem \ref{theo:Nash_top}): 

\begin{theo}\label{theo:main2}
  Let $E$ and $F$ be two prime divisors over the origin of $\CC^2$. The Nash-adjacency 
    $\overline{N}_E \supset \overline{N}_F$ only depends on the combinatorial 
    type of the pair $(E,F)$.
\end{theo}

The proof involves techniques that fall outside Algebraic Geometry,
and it is based on the proof of a result of the first author from \cite{Bob}. 
However, there are a few improvements 
and modifications which motivated us to include a detailed version here.

This theorem has a few surprising consequences that we now summarize.

If we have a Nash adjacency $\overline{N}_E \supsetneq \overline{N}_F$, 
then $codim(\overline{N}_E)< codim(\overline{N}_F)$. As a corollary of the
last theorem we have the following: if the Nash-adjacency 
$\overline{N}_E \supset \overline{N}_F$ is true, then 
$\cap_t \overline{N}_{E_t} \supseteq \cup_s \overline{N}_{F_s}$ for certain infinite families of
divisors $E_t$ and $F_s$ such that the pairs $(E_t,F_s)$ have the same combinatorial 
type as $(E,F)$ for any value of the parameters $(t,s)$. This imediately improves the codimension bound in an unexpected 
way: as a special case of a general theorem of Ein, Lazarsfeld and Musta\c{t}\u{a} \cite{ELM}, 
we know that given a prime divisor $E$, 
the codimension $codim(\overline{N}_E)$ of $\overline{N}_E$ in the arc space of $\CC^2$ 
equals its \emph{log discrepancy} $a_E(\CC^2)$ 
(its coefficient in the relative canonical divisor 
$+1$). Using the previous corollary, we improve in most cases the existing bound $a_E(\CC^2)<a_F(\CC^2)$ for the log discrepancy. Namely, we show that 
if $E$ is Nash-adjacent to $F$, then: 

\begin{equation}
    \label{local_intro}
         a_E(\CC^2)<a_F(\CC^2)-b,
\end{equation}
where $b$ depends on the combinatorial type of the pair $(E, F)$
and in which often $b>0$ (see Corollary~\ref{cor:codim}).

We also state some conjectures in Section \ref{sec:cyl}, relating the generalized 
Nash problem stated above to the existence of inclusions between maximal divisorial  
sets in the space of arcs associated with a valuation $q\cdot \nu_E$, as defined in \cite{ELM} 
or \cite{I_max}. We are convinced that the solution of these conjectures is crucial for the final solution of the
generalised Nash problem.

A basic fact, proved by Pl\'enat \cite{Cam} as a generalization of Reguera's 
 \cite[Thm. 1.10]{Reg1}, which concerned only the case of rational surface singularities, 
 is that one has the implication: 
     $$\overline{N}_F \subset \overline{N}_E\Longrightarrow \nu_E \leq \nu_F.$$
  This implication, which we call the \emph{valuative criterion} for Nash adjacency,  
  has been used as a main tool for the investigation of Nash's Conjecture until the works 
  \cite{Bob}, \cite{tesis} and \cite{Nash_surfaces}.  It has been the essential ingredient in the 
  resolution of some cases (see \cite{SandMR}, \cite{CamPP} and \cite{CamPPP}). 
  However,  Ishii~\cite{I_max} showed with an 
  example that even in the smooth surface case the converse implication is not true. 
  The second set of results of this paper (which for expository reasons appear in the course 
  of the paper before the results mentioned previously in this introduction) 
  concerns the investigation of the geometric meaning of the inequality $\nu_E \leq \nu_F$. 
  For this purpose there is no
  need to restrict to prime divisors, although for the case of prime divisors we often get
  stronger results.

In Proposition \ref{prop:val} of Section \ref{sec:defirst}, we prove a characterization of the 
inequality $\nu_E\leq \nu_F$ in terms of finitely many inequalities $\nu_E(h)\leq \nu_F(h)$ 
with  functions $h$ only depending on $E$.
When $E$ and $F$ are prime divisors, we see in Section \ref{sec:speprime} 
that this is equivalent to an even more 
reduced number of inequalities, with $h$ varying only among the approximate roots 
of a function related to $E$. 

Using the previous characterization, we clarify the geometric meaning of the inequality 
$\nu_E \leq \nu_F$ in the 
following way. Let $F=\sum_ib_iD_i$ be a divisor in a model $S$ of 
$(\CC^2,0)$ with exceptional locus $Exc_S=\bigcup_iD_i$. 
We say that a function $f$ is \emph{weakly associated with $F$ in $S$} 
if $b_i=\widetilde{V(f)}\cdot D_i$ where $\widetilde{V(f)}$ is the strict transform in $S$ of $V(f)$, 
the zero set of $f$. 
Furthermore, we say that $f$ is \emph{associated} to $F$ if $\widetilde{V(f)}$ is a disjoint union of 
smooth irreducible curve germs and if for every $i$, $b_i$ of them are transversal to the 
corresponding $D_i$, at distinct smooth points of $Exc_S$.  

Then, our result is the following (see Theorem \ref{theo:val_def}):
	
\begin{theo} \label{theo:main1}
Let $E$ and $F$ be divisors over the origin of  $\CC^2$ and let $S$ 
 be the minimal model containing the divisor $E+F$. The following are equivalent: 
 \begin{enumerate}[(a)]
			\item  $\nu_E\leq \nu_F$.
			\item there exists a deformation $G=(g_s)_s$ with $g_0$ 
			    weakly associated with $F$ in $S$ and $g_s$ weakly associated 
			    with $E$ in $S$, for $s\neq 0$ small enough.
			\item there exists a deformation $G=(g_s)_s$ with $g_0$ associated 
			     with $F$ in $S$ and $g_s$ associated with $E$ in $S$, 
			      for $s\neq 0$ small enough.  
                            \item there exists a linear deformation $(g_s = g_0+sg)_s$ with $g_0$ 
                                 associated with $F$ in $S$ and $g$ and $g_s$ associated with $E$ 
                                 in $S$, for $s\neq 0$ small enough.  
\end{enumerate}
\end{theo}
		
The equivalence $(c)\Leftrightarrow (d)$ was proved in different language 
by Alberich and Roe in \cite{AR}. See also 
Theorem~\ref{theo:val_def_prim} for an improvement when $F$ is a prime divisor. 

Let $E=\sum_ia_iE_i$ and $F=\sum_jb_jF_j$ be divisors in a 
model $S\to \CC^2$. Let $(f_s)_s$ be a curve deformation such that 
$f_0$ is associated with $F$ in $S$ and $f_{s \neq 0}$ is associated with $E$ in $S$. 
For $s\neq 0$ small enough, the germ $V(f_s)$ 
is a curve singularity with constant topological type. The topological type only jumps 
for $s=0$. Following Arnold's classical terminology, we say that
the deformation $(f_s)_s$ 
exhibits an \emph{adjacency} from the topological type of $V(f_{s \neq 0})$ 
to the topological type of $V(f_0)$ 
(note that in this paper, when we speak about the topological type of 
a curve singularity, we mean its embedded topological type).

As the reader may expect, not all of Arnold's adjacencies of topological types of plane curve singularities can be 
realized by curve deformations as above. Let us explain why. 
		
In a deformation $(f_s)_s$ giving rise to an Arnold adjacency, the positions of the 
free infinitely near points of the plane curve singularity $V(f_{s\neq 0})$ move with 
the parameter $s$ (see Section \ref{sec:not}). By contrast, 
the special property (b) above implies  
that the free points do not move with $s$. We say that such deformations  
\emph{fix the free points}. They correspond to the deformations of Alberich and 
Roe's paper ~\cite{AR} whose general member is associated with a cluster. 
		
In Section \ref{sec:def_fix}, we use Theorem~\ref{theo:main1} and its 
corollaries to show that most of 
the known Arnold adjacencies (see \cite{Arn}, \cite{Bri}, \cite{Saito} and \cite{BBB}) can be 
realized by deformations fixing the free points. 
We provide a complete list of adjacencies fixing the free points between all simple, unimodal and bimodal singularities 
with Milnor number up to $16$, and compare it with the list of usual Arnold adjacencies. This
improves the study carried in~\cite{AR} and allows to recover in a fast way many 
adjacencies from \cite{Arn}, \cite{Bri}, \cite{Saito}.

Using these results, in Remark~\ref{rem:alg} we sketch an algorithm 
which, starting from a topological type, gives back many topological types 
Arnold adjacent to the given one.
We expect that this algorithm produces a good portion of all the topological 
types Arnold adjacent to the given one.

In Section \ref{resdelta}, using the technique developed in this paper 
and a result of Ishii \cite{I_max} 
for the case of deformations whose general
member has only one Puiseux pair,  we obtain 
information about $\delta$-constant deformations: 

\begin{prop}
Let $E$ be a prime divisor which is the last divisor appearing in the minimal embedded 
resolution of an irreducible plane curve singularity with one Puiseux pair. 
Let $F$ be any other prime divisor. The following properties are equivalent:
\begin{enumerate}
\item there is a $\delta$-constant deformation $(f_s)_s$ such that $f_0$ is associated 
     with $F$ and $f_{s\neq 0}$ is associated with $E$.
\item the inequality $\nu_E\leq \nu_F$ holds (and hence all the equivalent formulations 
     predicted in Theorem \ref{theo:main1}).
\end{enumerate}
  If the previous properties are satisfied, then there exists a $\delta$-constant 
     deformation whose generic topological type is given 
      by a curve whose minimal embedded resolution hits $E$ and whose special curve 
      lifts transversely to $F$.
\end{prop}

As far as we know, this is the first concrete criterion allowing to find topological types 
in the $\delta$-constant stratum of a given plane curve singularity.

\medskip
We are now in a position to return to the case in which $E$ and $F$ are prime and compare the geometric meanings of the 
inequality $\nu_E\leq\nu_F$ and of the inclusion $N_F\subseteq N_E$.

Using results of Reguera \cite{Reg3} 
and of the first author \cite{Bob}, the generalized Nash problem can be reformulated as follows 
in the case of normal surfaces: 

\medskip
\emph{
Given $E$ and $F$ as before, does there exist a \emph{wedge} 
(a germ of a uniparametric family $(\alpha_s)_{s \in (\CC,0)}$ of arcs) 
such that $\alpha_0\in \overline{N}_F$ and $\alpha_{s\neq 0} \in \overline{N}_E$ 
and furthermore the liftings of $\alpha_0$ and $\alpha_{s\neq 0}$ to the minimal 
models $S_F$ and $S_E$ of $F$ and 
$E$ respectively are transversal to the corresponding exceptional divisors?}
\medskip
  
In the case of $X=\CC^2$, with coordinates $(x,y)$, the image in $\CC^2_{x,y}\times \CC_s$  of a 
wedge $\alpha:\CC_t\times \CC_s  \to \CC^2_{x,y}$ 
given by $(\alpha(t,s),s)$ has an implicit equation $f(s,x,y)=0$ 
which induces a deformation of a plane curve singularity (define 
$f_s(x,y):=f(s,x,y)$). 
It is clear that $f_0$ is 
associated with $F$ and that $f_{s \neq 0}$ is associated with $E$. 
Since the deformation admits a parametrization in family
(given by $\alpha$), it satisfies the additional property of being a $\delta$-constant deformation. 
In fact, due to a classical theorem of Teissier (see Theorem~\ref{theo:Teissier_delta}), 
the converse is true: if there is a $\delta$-constant
deformation $f_s$ such that $f_0$ is associated with $F$ and that 
$f_{s \neq 0}$ is associated 
with $E$, then there is a uniparametric 
family $(\alpha_s)_s$ as above, and by results of the first author~\cite{Bob} 
we have the inclusion $\overline{N}_F\subseteq \overline{N}_E$.

Thus we find that the geometric difference between the inequality  $\nu_E\leq\nu_F$ and the inclusion $N_F\subseteq N_E$
is the $\delta$-constancy of the associated deformations. More precisely:

\begin{prop}
Let $E$, $F$ be prime divisors over $O$ such that the inequality  $ \nu_E\leq\nu_F$ holds. 
The inclusion $N_F\subseteq N_E$
is satisfied if and only if there exists a \emph{$\delta$-constant} deformation $G=(g_s)_s$ with $g_0$ associated 
with $F$ in $S$ and $g_s$ associated with $E$ in $S$, for $s\neq 0$ small enough.
\end{prop}

In Section \ref{sec:ex} we show an example where condition $\nu_E\leq\nu_F$  
holds but the inclusion $\overline{N}_E \supset \overline{N}_F$ does not. 
Ishii gave already such an example in~\cite{I_max}. 
But in the case of her example, $E$ is a free divisor and one can disprove the inclusion 
$\overline{N}_E \supset \overline{N}_F$ by the 
log discrepancy criterion mentioned above.
By contrast, our example concerns two satellite divisors and the 
inclusion $\overline{N}_E \supset \overline{N}_F$
cannot be discarded even by the enhanced log discrepancy criterion explained above in this introduction.

 \medskip
{\bf Acknowledgements:}
We thank D. Van Straten and J. Stevens for having pointed out reference \cite{BBB}. 
We thank S. Gusein-Zade and 
V. Goryunov for a discussion on classical adjacencies.  
The second author wants to thank Felix Delgado for very useful conversations and 
the third author thanks Camille Pl\'enat for her remarks on a previous version of the paper. 
We also thank the referee for his careful revision and useful suggestions.

The first author is partially supported by IAS and by ERCEA 615655 NMST Consolidator Grant, MINECO by the project reference MTM2013-45710-C2-2-P, by the Basque Government through the BERC 2014-2017 program, by Spanish Ministry of Economy and Competitiveness MINECO: BCAM Severo Ochoa excellence accreditation SEV-2013-0323 and by Bolsa Pesquisador Visitante Especial (PVE) - Ciencias sem Fronteiras/CNPq Project number:  401947/2013-0. 

	The first and second author are partially supported by the ERC Consolidator Grant NMST,  
	by Spanish MICINN project MTM2013-45710-C2-2-P and by the project SEV-2011-0087 Apoyo a Centros y Unidades de Excelencia Severo Ochoa. The first author is also partially supported by a CNPq project, Science without borders, Project nº 401947/2013-0. 
	The second one was also supported by IMPA (Rio de Janeiro).
 The second and third author were partially supported 
   by Labex CEMPI (ANR-11-LABX-0007-01). 
	The third author was also partially 
   supported by the grant ANR-12-JS01-0002-01 SUSI.

\section{Background and terminology}\label{sec:not}

In this section we introduce terminology about divisors, curves,  
infinitely near points and combinatorial types of divisors to be used throughout the paper. 
The reader can find further details in \cite{Wall}, \cite{dJPf} or \cite{Casas}.
\medskip

We work with $(\CC^2,0)$, seen as a germ of complex analytic surface, 
with local $\CC$-algebra $\calO$, whose maximal ideal is denoted 
$\frak{m}$. If $h \in \frak{m}$, we denote by $V(h)$ its vanishing 
locus, seen as a germ of effective divisor on $\CC^2$. If $h$ is 
irreducible, then $V(h)$ is called a \emph{ branch}. In general 
we say simply that $V(h)$ is a \emph{curve singularity}, or even a \emph{curve}, 
if no confusion is possible with a global curve.  

We denote by $m_0(h) = m_0(V(h))$ the \emph{multiplicity} at $0$ of $h$ or of the curve 
$V(h)$ it defines. 

If $C$ and $D$ are two curves, 
their \emph{intersection number} at a point $p$ (in general, $p$ will be an infinitely 
near point of $0$, as defined below) is denoted by:
   \[ I_p(C,D) \in \NN \cup \{+ \infty\}. \]
One gets $+ \infty$ if and only if $C$ and $D$ share irreducible components.

The generalized Nash problem concerns prime divisors above $0$. 
But, since the valuative order and some results concerning it in Section \ref{sec:val} 
can be extended to divisors which are not necessarily prime, we introduce our terminology 
in that generality.  

A \emph{model} is any smooth surface $S$ obtained from $(\CC^2,0)$ by a finite sequence of 
blow-ups of points.  If $\pi : S \to \CC^2$ is the associated morphism, we denote by 
$Exc_S := \pi^{-1}(0)$ the corresponding \emph{ exceptional divisor}. 
If $S, T$ are two models and $\pi_S$ and $\pi_T$  are the associated morphisms, 
we say that $T$ \emph{dominates} $S$ if 
there exists a morphism $\psi : T \to S$ such that $\pi_S \circ \psi = \pi_T$. 

Any point of the exceptional divisor of some model is called an \emph{infinitely near point of $0$}. 
Two such points in different models are considered to be equivalent if the canonical 
bimeromorphic mapping 
between the models which contain them restricts to a biholomorphism in their neighborhoods.
In the sequel, we will call such equivalent classes of infinitely near points simply \emph{points}.  

In a process of blowing up points starting from $0$, 
if a  divisor $E$ is created by blowing up a point $p$, we may also denote that divisor  by $E_p$ or that point by $p_E$. If a point $q$ lies on 
the strict transform of $E_p$ on some model, we say that $q$ \emph{is proximate to} $p$.  

Recall that in a blowing up process, an infinitely near point of $0$ 
can be either \emph{satellite}, if it is  
a singular point of the exceptional divisor of the model on which it lies, that is, 
if it is the intersection point of two components of this divisor, 
or \emph{free}, if it is a smooth point of it. We say also that a divisor $E$ is \emph{satellite/free} 
if $p_E$ is so. 

In the sequel, a \emph{prime divisor} (\emph{above $0$}, which we won't precise 
if there is no risk of confusion) will mean any irreducible component of the exceptional 
divisor of any model, two such divisors being \emph{equivalent} if there exists a third model 
on which their strict transforms coincide. In other words, they are equivalent if and only 
if they define the same divisorial valuation on the local ring $\calO$. 
We will make an abuse of language and talk about prime divisors also 
for equivalence classes of divisors according to this relation.

More generally, a  \emph{divisor}  $D$ will be an element of the free abelian 
semigroup generated by the prime divisors. That is, $D = \sum_{i=1}^n a_i D_i$, 
where $a_i \in \NN$ and $D_i$ is a prime divisor (that is, our ``divisors'' will 
be effective divisors in the usual terminology of algebraic geometry). 
Whenever $a_i \neq 0$, we say that $D_i$ is a \emph{component of $D$}.

Given a divisor $D$, we say that it \emph{appears} on a model $S$ if 
its components are equivalent to irreducible components of the exceptional 
divisor $Exc_S$ of $S$.  
Among all the models on which $D$ appears, there is a minimal one under the 
relation of domination of models. We call it the \emph{minimal model of $D$}.

  \begin{defn}  \label{def:domrel}
      Let $E$ and $F$ be two prime divisors. We say that $F$ \emph{dominates} $E$, 
    and we write $F \geq_d E$, if and only if the minimal model of $F$ 
    dominates the minimal model of $E$.
   \end{defn}

\begin{defn}\label{def:assoc}  
     Given a divisor $D=\sum_{i=1}^n a_i D_i$ with $D_i$ irreducible and 
     $a_i\geq 0$ that appears in a model $S$, then we say that $h\in\calO$ 
    \emph{ is associated with $D$ in} $S$ if: 
          $$h=\prod_{j}h^{(j)}$$ 
    with $h^{(j)}$ irreducible factors 
    and for every $i \in \{1, \dots , n\}$, we have that $a_i$ of them are functions defining pairwise 
    disjoint \emph{curvettas} of $D_i$ in $S$ (that is, branches whose strict transforms 
    on $S$ are smooth and hit $D_i$ transversely in different smooth points of $Exc_S$). 
      We say that $h$ and $V(h)$ are, respectively, a  
      \emph{function and a curve associated with $D$ in $S$}. We will denote a 
      function associated with $D$ by $h_D$.
\end{defn}

Recall that the final prime divisors of the minimal embedded resolutions of branches are 
satellite. Conversely, given any satellite prime divisor $D$, it is the final divisor of the 
minimal embedded resolution of any branch associated with it.

\emph{The combinatorial type of a (not necessarily satellite) prime divisor} $D$ 
may be encoded in the following 
equivalent ways (see more details in \cite[Chapter 5]{dJPf} and \cite[Section 3.6]{Wall}):
 \begin{enumerate}
          \item  by the dual graph, which we denote by $\calG_D$, of the exceptional 
              divisor of the minimal model of $D$, 
             each vertex being weighted  by the self-intersection of the corresponding 
             component in this minimal model; 
         
         \item by the character free/satellite of the finite set of points which have to be blown 
            up in order to reach the minimal 
             model of $D$, enriched with the information about the proximity binary relation; 
             this information can be represented for example by the Enriques diagram 
             (see \cite[Chapter 3]{Casas}); 
       
           \item the \emph{multiplicity sequence}. This sequence associates 
               to the height $k$ above $O$ of each infinitely near point $p_k$ which is blown up in order 
               to reach $D$, the multiplicity at $p_k$ of the strict transform of $V(h_D)$ for some    
               associated function $h_D$.
    \end{enumerate}
    
	The \emph{combinatorial type} of a not necessarily prime divisor $D$ 
can be encoded in analogous ways, by keeping also the coefficients in 
$D$ of each component of the minimal model of $D$. For example, in the dual graph $\calG_D$ 
we mark the vertices which 
correspond to the components  of $D$ and we add their coefficients in $D$ 
as a second decoration of those vertices of $\calG_D$. 

We write $D\equiv D'$ to say that $D$ and $D'$ have the same combinatorial type. 

Observe that a divisor $D$ is not determined by its combinatorial type, 
 even up to isomorphisms 
above $(\CC^2,0)$. The existence of free points in the resolution process 
(and the points to be blown up immediately 
after $0$ are always free points) gives moduli in the family of divisors with the same 
combinatorial type. In~\cite{BobTesis},~\cite{Roe} moduli spaces of divisors with the same combinatorial type 
are constructed.
   
Note that for a free prime divisor $E$, the embedded topology of a curve $V(h)$ 
associated with it (in some model) contains 
less information than the combinatorial type of $E$. Indeed, the combinatorial type 
of a prime divisor $E$ is codified by the minimal embedded resolution of any associated curve $V(h)$ and the
number of free points we have to blow up after resolving such a curve $V(h)$ in order to get the minimal model of $E$. 
That is, it is codified adding to the multiplicity sequence of the associated curve $V(h)$ 
as many $1$'s as free points we have to blow up after resolving it, in order to get $E$.

Given a pair of prime divisors $E$ and $F$, we define the \emph{contact order} $Cont(E,F)$ between 
them as the number of common points $p_i$ they share in the process of finding 
their minimal models. Said slightly differently, it is the number of blow ups that we perform 
before the strict transforms of associated curves $V(h_E)$ and $V(h_F)$ to $E$ and $F$ respectively in a model of $E+F$
separate. For instance, assume that 
both $E$ and $F$ are different from $E_0$. 
Then their contact order is $1$ if and only if the branches $V(h_E)$ and $V(h_F)$ 
have different tangent lines at $0$. 

Using this notion, we can say that the combinatorial type of a non-prime divisor is given by the combinatorial types of its prime components, their multiplicities in the 
divisors and the contact orders between them.

Given two divisors $E$ and $F$, we will talk about the 
\emph{combinatorial type} of the pair $(E, F)$ as the combinatorial type 
 of both $E$ and $F$ taken separately, to which we add the information of all the contact 
 orders between all their prime components. We write $(E,F)\equiv (E',F')$ 
 to say that these two pairs have the same combinatorial type. Equivalently, we can consider the dual graph of the exceptional divisor of the minimal model where all the components of $E$ and $F$  appear and mark differently the components of $E$ and $F$ and 
 their corresponding multiplicities.


\section{The valuative partial order}\label{sec:val}

In this section we introduce the \emph{valuative domination}  partial order 
$\leq_{\nu}$ on the set of non necessarily prime divisors above $O$ and explore 
several properties of  it. 

In the first subsection we show that the valuative inequality $E \leq_{\nu} F$ 
is equivalent to
a  finite number of inequalities 
$\nu_E(h)\leq\nu_F(h)$, for well-chosen functions $h$ depending only on $E$. 
In particular, we see in this way that the valuative domination between two divisors $E$ and $F$ 
only depends on the combinatorics of the pair $(E,F)$.

In Subsection \ref{sec:speprime} we prove that when $E$ and $F$ are \emph{prime}, 
then in order to show the valuative inequality $E \leq_{\nu} F$,  
it is enough to check those inequalities 
when $h$ varies among the approximate roots of a function 
which defines a curvetta of $E$.

In Subsection \ref{sec:def}  we prove a characterization of  the pairs of divisors $(E,F)$ such that 
$E \leq_{\nu} F$  in terms of the existence of certain types of deformations, which we call 
{\em deformations fixing the free points}  
(see Definition~\ref{def:def_fix} and Theorems \ref{theo:val_def} and \ref{theo:val_def_prim}). 
This recovers and generalizes in a different language results of Alberich and Roe~\cite{AR}. 

In Subsection \ref{sec:def_fix} we see that these deformations realize many classical 
adjacencies of topological types, testing the previous criterion  
on Arnold's list of adjacencies. In Remark \ref{rem:alg} we sketch an algorithm which, 
given the embedded topological type of a branch, finds all the Arnold 
adjacent topological types obtained by deformations fixing the free points.
\medskip


\subsection{Definition and first properties of the valuative partial order}   \label{sec:defirst}

$\: $
\medskip

Given a prime divisor $E$, we denote by $\nu_E$ the divisorial valuation associated with it. 
Let $\pi_S :  S\to \CC^2$ be a model in which $E$ appears. 
Given $h\in \calO$, the value $\nu_E(h)$ is the vanishing order along $E$ 
of the total transform $\pi_S^*(h)$ of $h$ on $S$. If the divisor $E$ is not prime, 
then it may be written  as a sum $\sum_i a_i E_i$ for some prime $E_i$'s and 
non-negative integers $a_i$'s.
We define then: 
       $$\nu_E:=\sum_ia_i\nu_{E_i}.$$     
Note that, whenever $E$ has at least two components, $\nu_E$ is no longer a valuation. 
 
The value $\nu_E(h)$ can be also computed as the intersection multiplicity 
of $V(h)$ with a curve $V(h_E)$ associated with $E$ in $S$ (see \ref{def:assoc}),
 whenever $h_E$ is chosen such that 
the strict transforms of $V(h_E)$ and $V(h)$ in the minimal model of $E$ do not meet. 
An immediate consequence of this fact is:

  \begin{lem}  \label{lem:exchange}
      For any two divisors $E$, $F$ above $0$, one has $\nu_E(h_F) = \nu_F(h_E)$ for $h_E$ 
       and $h_F$ generic associated functions to $E$ and $F$ in the minimal model of $E+F$. 
  \end{lem}

\begin{defn} \label{defn:valord}
     We define the \emph{valuative domination partial order} $\leq_\nu$ among divisors,   
     by saying that $E\leq_\nu F$ if and only if $\nu_E(h)\leq \nu_F(h)$ for all $h\in\calO$. 
     We also write $\nu_E\leq\nu_F$,  where $\nu_E$ and $\nu_F$ are 
     seen as functions defined on the set $\mathcal{O}$ with values in $\NN$. 
     We say then that \emph{ $F$ valuatively dominates $E$}  or that 
      \emph{ $E$ is valuatively dominated by $F$} and we call $E\leq_\nu F$ the 
      \emph{valuative  inequality}.
\end{defn}

\begin{rem}
\label{rem:dom-val}
Let $E$ and $F$ be two prime divisors. If $F$ dominates $E$ in the sense of 
Definition~\ref{def:domrel},  then $F$ valuatively dominates $E$.
\end{rem}

The only non-trivial axiom of partial order is the antisymmetry, which is verified  
in Lemma \ref{lem:antisym}.

Consider two curves $V(f)$ and $V(g)$. Let $\{x_i\}_{i\in I}$ be the infinitely near points 
which we blow up to obtain an embedded resolution of $V(fg)$.
Let $\tilde{f}_i$ and $\tilde{g}_i$ be the local equations of the strict transforms of 
$V(f)$ and $V(g)$ at $x_i$. 
Denote $m_i^f=m_{x_i}\tilde{f}_i$ and $m_i^g=m_{x_i}\tilde{g}_i$. Then, we have 
(see for example \cite{dJPf}):
\begin{equation}
\label{eq:mult}I_O(V(f),V(g))=\sum_{i\in I} m_i^f\cdot m_i^g.
\end{equation}

The following two lemmas are immediate consequences of this equality: 

\begin{lem}\label{lem:compar}
    Let $E$ be a divisor above $0$ and $h_E$ a function associated with $E$ 
    in the minimal model $S_E$ of $E$. 
    Then for any $h \in \mathfrak{m}$, one has:
        \[   \nu_E(h) \leq I_0(V(h_E), V(h)) \]
     with equality if and only if the strict transforms of $V(h_E)$ and 
     $V(h)$ do not meet in $S_E$. 
\end{lem}
  
\begin{lem}\label{lem:nu_ind} 
        The value of $\nu_E(h)$ only depends on the 
			combinatorial type of the exceptional divisor of the minimal 
			model which both gives an
			embedded resolution of $V(h)$ and makes $E$ appear, endowed 
			with the marking of the divisor $E$ 
			and of the divisor met by the strict transform of $V(h)$ in this model.

\end{lem}
\begin{proof}
  The combinatorial type we refer to may be described by the following decorated 
        graph, associated to the minimal embedded resolution in which $E$ appears:  
        \begin{itemize}
             \item let $\calG(E,h)$ be the dual graph of the exceptional divisor; 
             \item weight each vertex by the self-intersection of the corresponding divisor; 
             \item if an irreducible component of the strict transform of $V(h)$ meets an 
                 irreducible component of the exceptional divisor, attach an arrow to the 
                 vertex which corresponds to the strict transform; 
             \item mark the vertices corresponding to the components of $E$ with the label ``$E$''.
        \end{itemize}

 Choose a function $h_E$ such that the strict transforms of $h_E$ and of $h$ 
 do not meet in the minimal resolution considered
 in the statement. Then by Lemma~\ref{lem:compar} we have the equality 
 $\nu_E(h)=I_0(V(h_E), V(h))$. By the usual formulas 
 controlling the behaviour of intersection multiplicity of plane branches under blow up, 
 it is clear that the combinatorics 
 of the decorated graph $\calG(E,h)$ determine $I_0(V(h_E), V(h))$. 
\end{proof}

\begin{lem}\label{lem:val_top} 
       If $(E,F)\equiv(E',F')$, that is, if the two pairs have the same combinatorial type, then: 
          \begin{equation}\label{eq:val_top}
                    E\leq_\nu F \ \Leftrightarrow\ E'\leq_{\nu}F'.
          \end{equation}
\end{lem}
\begin{proof}

Consider the minimal blowing up sequence $\pi_{E,F}:X_{E,F}\to\CC^2$
 such that every component of $E$ and $F$ appear in the exceptional divisor.
 Consider the graph $\calH(E,F)$ to be the dual graph of the exceptional divisor; 
 weight each vertex by the 
 self-intersection of the corresponding prime divisor; mark the vertices corresponding to prime components in $E$ by a 
 label ``$E$'' and the vertices corresponding to prime components in $F$ by a label ``$F$'' (observe that a vertex can 
 carry the 2 labels).
 
 Given any function $h \in \mathcal{O}$, we consider 
 the strict transform $\widetilde{V(h)}$ to $X_{E,F}$ of the curve defined by it. 
 To any vertex of $\calH(E,F)$ we give a secondary weight given by the intersection number of the corresponding prime 
 divisor and the strict transform $\widetilde{V(h)}$. Observe that the graph $\calH(E,F)$ together 
 with the second weighting determine $\nu_E(h)$ and $\nu_F(h)$.
 
 Moreover, given any non-negative secondary system of weights 
 of the vertices, there exists a function $h$ which gives rise to those weights by the procedure above. It 
 can be constructed as follows: in the model $X_{E,F}$ we consider a curve that intersects each divisor $E$ with the 
 intersection multiplicity prescribed by the weight (this can be done because the exceptional 
 divisor has only normal crossings).
 Blow down the curve and consider an equation $h$ for it.

 The lemma holds because we have the combinatorial equivalence $(E,F)\equiv(E',F')$ if and only if the 2 decorated 
 graphs $\calH(E,F)$ and $\calH(E',F')$ are isomorphic, and because the possible pairs ($\nu_E(h),\nu_F(h)$) are 
 determined from the 
 non-negative weights explained above.

 \end{proof}

 The following lemma is obvious for prime divisors, but for the general case we include a short proof.

\begin{lem} \label{lem:antisym}
    If $E$ and $F$ are divisors above $0$ and $E \leq_\nu F \leq_\nu E$, 
       then $E = F$, that is $\leq_\nu$ is a partial order relation. 
\end{lem}
\begin{proof}
       Given $E$ and $F$ with $E \leq_\nu F \leq_\nu E$, we have to check that the 
       multiplicity sequences of $E$ and $F$ are the same and 
        that the same sequence of infinitely near points is blown up to obtain their minimal models. 

       Take a branch $V(h)$ with contact order $1$ with both $V(h_E)$ and $V(h_F)$, 
       for $h_E$ and    $h_F$ associated functions to $E$ and $F$. Then, 
       $\nu_E(h)=\nu_F(h)$ implies that the multiplicities of $h_E$ and $h_F$ are the same. 
       To see this implication we use Lemma \ref{lem:compar} and (\ref{eq:mult}).

       Take another branch $V(h)$ with contact order 2 with some branch of $V(h_E)$ 
       and no higher contact order with $V(h_F)$. We impose $\nu_E(h)=\nu_F(h)$. 
       Using $m_0(h_E)=m_0(h_F)$, we conclude that $V(h)$ has also contact order 
       $2$ with some branch of $V(h_F)$ and that the strict transforms of $V(h_E)$ and 
       $V(h_F)$ at that second contact point have the same multiplicity. 

       Using functions $h$ with increasingly contact order with the different branches of 
       $V(h_E)$, we inductively conclude that $V(h_E)$ and $V(h_F)$ share all the 
       infinitely near points above the origin $0$ and that their multiplicities are the same 
       at all of them.   
\end{proof}
 
To check whether two divisors are comparable for the valuative order, 
we can use the following criterion: 

\begin{prop}\label{prop:val}
    Let $E\neq F$ be divisors above $0$. Denote by 
    $S$ the minimal model of $E + F$. Given a divisor $D$ above $0$ in the model $S$, 
    $h_D$ denotes an associated function in $S$. The following are equivalent: 
\begin{enumerate}
   \item $E\leq_\nu F$, that is  $\nu_E(h)\leq \nu_F(h)$ for all $h\in\calO$. 
   \item $\nu_H(h_E)\leq \nu_H(h_{F})$ for all prime divisors $H$. 
   \item $\nu_H(h_E)\leq \nu_H(h_F)$ for all prime divisors $H$ appearing 
     in  $S$.
   \item $\nu_E(h_H)\leq \nu_F(h_H)$ for all prime divisors $H$ appearing in $S$. 
    \item $\nu_E(h_H)\leq \nu_F(h_H)$ for all prime divisors $H$ appearing in  
     the minimal model of $E$. 
   \item $\nu_H(h_E)\leq \nu_H(h_F)$ for all prime divisors $H$ appearing in  
      the minimal model of $E$. 
\end{enumerate}
\end{prop}

\begin{proof}  The equivalences 
$(1)\Leftrightarrow(2)$, $(3)\Leftrightarrow(4)$ and $(5)\Leftrightarrow(6)$ are immediate 
consequences of Lemma \ref{lem:exchange}. The implications 
$(2)\Rightarrow(3)\Rightarrow (6)$ are immediate.

\medskip

Let us prove that $(6)\Rightarrow(2)$. Consider any prime divisor $H$. Let $\{p_j\}_{j\in J}$  
be the infinitely near points that we blow up to obtain
the minimal model of $H$. The index set $J$ is ordered by saying that 
$j>i$ if $p_j$ is infinitely near $p_i$. This  order is total since $H$ is prime. 
Let $H_j$ be the divisor 
which appears after blowing up $p_j$. Let $J'\subset J$ be the subset consisting 
of the indices $j' \in J$ 
such that $H_{j'}$ appears in the minimal 
model of $E$. It is enough to check that $\nu_{H_j}(h_E)\leq \nu_{H_j}(h_F)$ for all $j \in J$.  
The inequality is true by hypothesis for $j\in J'$. Assuming it is true up to $j$, 
we check it for $j+1$, the immediate successor of $j$ for the previous total order 
(provided it does not belong to $J'$):  
\begin{itemize}
        \item if $p_{j+1}$ is a satellite point, then it is the intersection point of $H_{j}$ with  $H_{k}$ 
           for some $k<j$ and then:  
					 $$\nu_{H_{j+1}}(h_E)=\nu_{H_j}(h_E)+\nu_{H_k}(h_E),$$
					 $$\nu_{H_{j+1}}(h_F)=\nu_{H_j}(h_F)+\nu_{H_k}(h_F) +m_j^F,$$
					where $m_{j}^F\geq 0$,
     \item if $p_{j+1}$ is a free point, then:  
				    \begin{equation}\label{eq:free}
				         \nu_{H_{j+1}}(h_E)=
                                                \nu_{H_j}(h_E)\leq \nu_{H_j}(h_F)+m_j^F= 
                                                \nu_{H_{j+1}}(h_F)
                                          \end{equation}
             where $m_{j+1}^F$ denotes the multiplicity at $p_{j+1}$ of the strict transform 
             of $V(h_F)$ (which could be $0$). 
    
\end{itemize}
In both cases we get the desired inequality $\nu_{H_{j+1}}(h_E)\leq \nu_{H_{j+1}}(h_F)$.
\end{proof}

Note that each of the conditions $(3)-(6)$ gives a \emph{finite} number of inequalities 
to be checked.

\begin{cor}\label{cor:val_max_cont}
   Assume $\nu_E\leq \nu_F$ with $E$ and $F$ prime divisors. If the prime divisor  
    $F'$ is such that $F'\equiv F$ and $Cont(E,F')\geq Cont(E,F)$, then we have 
    $\nu_{E}\leq \nu_{F'}$. 
\end{cor}

\begin{proof} 
    By hypothesis and using the equivalence $(1)\Leftrightarrow(6)$ of Proposition 
    \ref{prop:val},  we have that $\nu_H(h_E)\leq \nu_H(h_F)$ for all prime divisors  
    $H$ in the minimal model of $E$. 
      Using formula (\ref{eq:mult}) and the hypothesis, it is clear that $\nu_H(h_F)\leq \nu_H(h_{F'})$. 
      Using the equivalence $(6)\Leftrightarrow(1)$ of Proposition \ref{prop:val},  we conclude. 
\end{proof}

Using the same type of arguments, Lemma  \ref{lem:nu_ind} and Lemma~\ref{lem:val_top}, 
we get also the following: 

\begin{cor}\label{cor:max}
     Let $E=\sum_ia_iE_i$, $F=\sum_j b_jF_j$ be two divisors above $O\in\CC^2$ such that 
    $E\leq_\nu F$. Assume that $E'=\sum_i a_iE_i'$ and $F'=\sum_jb_jF_j'$ 
    are such that $E_i \equiv E_i'$ for all $i$, $F_j \equiv F_j'$ for all $j$ and moreover 
    $E \equiv E'$, $F \equiv F'$. 
    If $Cont(E_i', F_j')\geq Cont(E_i,F_j)$ for all $i$ and $j$, then $E'\leq_\nu F'$. 
\end{cor} 

Observing the rules of computation of the divisorial valuations, we get the following corollary:  

\begin{cor}\label{cor:finite_val} 
     Given a divisor $F$, there exists a finite number of combinatorial types of pairs $(E,F)$ 
     such that $E\leq_\nu F$. 
\end{cor}

\begin{proof}

If $E\leq_\nu F$, then each of the prime components $E_i$ of $E$ satisfies the 
   analogous inequality $E_i \leq_\nu F$. 
Thus it is enough to assume
that $E$ is prime. In that case the sequence of blowing up centers
of the minimal model where $E$ appears is totally ordered.

Let $\{m_i^E\}_{i\in I_E}$ be the multiplicity sequence of a curve associated 
with $E$ in the minimal model of $E$. 
Let $g$, $h$ be  functions associated with $E$ in its minimal model, so that the 
strict transforms in this minimal model of
the curves defined by them are disjoint. In that case, by Lemma~\ref{lem:compar} and 
Formula~(\ref{eq:mult}) we have:

$$\nu_E(h)=I_0(V(g),V(h))=\sum_{i\in I_E} (m_i^E)^2.$$

The set $\{x_i\}_{i\in J}$ of infinitely near points that we need to blow up in order to obtain the minimal model $S$
where all components of  $F$ appear is finite. Let $f$ be a function 
associated with $F$ in $S$. Denote by $m_i^F$ the multiplicity of the 
strict transform of $V(f)$ at the point $x_i$. Choosing $f$ so that the strict 
transforms of $V(f)$ and $V(h)$ in the minimal model 
of $E\cup F$ are disjoint, by Lemma~\ref{lem:compar} and 
formula~(\ref{eq:mult}) we get:

$$\nu_F(h)=I_0(V(f),V(h))=\sum_{i\in I_E\cap J} m_i^F\cdot m_i^E.$$

Hence we have
\begin{equation}
  \label{ineqlocal}
     \nu_F(h)- \nu_E(h)
      =\sum_{i\in I_E\cap J} (m_i^F-m_i^E)\cdot m_i^E-\sum_{i\in I_E\setminus J} (m_i^E)^2.
\end{equation}

The inequality $\nu_E(h)\leq \nu_F(h)$ implies at once the inequalities 
   $m_i^E\leq m_i^F$ for any $i\in J\cap I_E$ and 
   that the cardinality of $I_E\setminus J$ is bounded by $\sum_{i\in J}(m_i^F)^2$. 
   Hence the cardinality of $I_E$ for 
     any $E$ satisfying $\nu_E\leq \nu_F$ is bounded
     by a universal bound only depending on the combinatorial type of $F$.

Moreover, the multiplicities $m_i^E$ are also universally bounded by bounds only 
    depending on the combinatorial type of $F$: otherwise, Formula~(\ref{ineqlocal}) 
    would contradict the inequality $\nu_E(h)\leq n_F(h)$.

Since the 
multiplicity sequences of the prime components of a given divisor and the 
orders of contact of its pairs of components 
determine the combinatorial type of the divisor, the corollary is proven.

\end{proof}

\subsection{The special case of prime divisors}\label{sec:speprime}

$\: $
\medskip

In this subsection we will show that Proposition \ref{prop:val} may be improved 
when both $E$ and $F$ are prime divisors (see Proposition \ref{prop:prime} below). 
We start by recalling several facts about the approximate roots of polynomials, 
whose proofs may be found in \cite{PPP}. In that reference one may also find 
explanations about the papers of Abhyankar and Moh 
in which this theory was developed. 
\medskip

In the sequel, if $C$ is a branch on $(\CC^2, 0)$, we say that 
$f \in \CC\{x\}[y]$ is a \emph{Weierstrass polynomial associated with $C$}  if:
  \begin{itemize}     
      \item  it is monic and irreducible; 
      \item it defines $C$ in some local coordinates $(x,y)$ on $(\CC^2, 0)$; 
      \item the $y$-degree $d_y(f)$  is equal to the multiplicity of $C$ at the origin.
    \end{itemize}
    
  Whenever the coordinate system $(x,y)$ is chosen 
  such that the $y$-axis is transversal  to $C$, 
  it determines a unique Weierstrass polynomial associated with $C$. In the sequel 
  we assume that such a coordinate system is fixed and that $f$ is the associated 
  polynomial. 
    
 If one computes the roots of $f$ 
 as Newton-Puiseux series in $x$, the sequence $\alpha_1 < \cdots < \alpha_g$ 
 of their rational \emph{characteristic exponents} is a topological invariant of $C$. 
 One has $g \geq 1$ if and only if $C$ is singular. Denoting 
 by $\beta_0 \in \NN^*$ the least common denominator of those numbers, 
 let us introduce Zariski's notation (see \cite{Zar}):
   \[ \beta_i := \beta_0 \cdot \alpha_i, \:  \forall \:   i \in \{1, ..., g \}. \]
 We say that $(\beta_0, \beta_1, ..., \beta_g)$ is the \emph{characteristic sequence} 
 of $C$ or of $f$. One has $d_y(f) = \beta_0 < \cdots < \beta_g$. Introduce also the integers:
    \[ e_i := gcd(\beta_0, ..., \beta_i) , \:  \forall \:   i \in \{0, ..., g \} \]
  and their successive quotients:
    \[n_i := \dfrac{e_{i-1}}{e_i}, \:  \forall \:   i \in \{1, ..., g \}. \]
 Those quotients are integers $\geq 2$, because one has the following sequence 
     of strict divisibilities: $1 = e_g | \cdots | e_1 | e_0 = 
 \beta_0 =   d_y(f)$. 
 Finally, define another sequence 
     $(\overline{\beta}_0, \overline{\beta}_1, ..., \overline{\beta}_g)$ 
 (which is the minimal generating sequence of the 
 semigroup of $C$, see \cite[Prop. 4.2]{PPP}):
   \[\overline{\beta}_0: = \beta_0, \: \: \:  \overline{\beta}_1 := \beta_1, \: \: \: 
    \overline{\beta}_{i+1} := n_i \overline{\beta}_i + \beta_{i+1} 
       - \beta_{i}, \:  \forall \:   i \in \{1, ..., g-1 \}. \]

 In general, when one has a divisor of the degree of a monic polynomial which is invertible 
 in the ring of coefficients of the polynomial, it is possible to associate to it canonically a new 
 polynomial (see \cite[Prop. 3.1]{PPP}):

 \begin{prop}
   Let $A$ be a commutative ring with unit and $P \in A[y]$ a monic polynomial 
   of degree $d >0$. Let $e >0$ be a divisor of $d$, which is moreover invertible 
   in $A$. Then there exists a unique monic polynomial $Q \in A[y]$ such that 
   $d_y(P - Q^e) < d- \displaystyle{\frac{d}{e}}.$
 \end{prop}

 This polynomial is called the \emph{$e$-th approximate root} of $P$. It is of degree 
 $\displaystyle{\frac{d}{e}}$.

 In particular, one may consider the $e_0, ... , e_g$-th approximate roots of $f$. 
 Denote them by $f_0, ..., f_g$ respectively. Therefore 
 $f_g = f$. Moreover (see \cite[Thm. 5.1]{PPP}):

 \begin{prop}
    The approximate roots $(f_k)_{0 \leq k \leq g}$ have the following properties: 
       \begin{enumerate}
           \item $d_y(f_k)  = \displaystyle{\frac{\beta_0}{e_k}}$ and 
               $I_0(V(f), V(f_k)) = \overline{\beta}_{k+1}$  
               (where, by convention, $\overline{\beta}_{g+1}= + \infty$); 
            \item the polynomial $f_k$ is a Weierstrass polynomial for $V(f_k)$ 
               and its characteristic sequence is \linebreak 
                 $\displaystyle{(\frac{\beta_0}{e_k}, ...,  \frac{\beta_k}{e_k})}$. 
       \end{enumerate}
 \end{prop}

 For us, the most important property of the previous sequence of approximate 
 roots of $f$ is the possibility to express canonically any other polynomial 
 $(f_0, ..., f_g)$-adically, similarly to the expansion of integers in a basis of numeration 
 (see \cite[Cor. 5.4]{PPP}):

\begin{prop}  \label{prop:maintool}
    Every (not necessarily monic) polynomial $h \in \CC\{x\}[y]$ may be expressed 
    uniquely as a finite sum of the form:
       \[ h = \sum_{(i_0, ..., i_g) \in J(h)} a_{i_0, ..., i_g} f_0^{i_0} \cdots f_g^{i_g} \]
    where: 
        \begin{itemize}
             \item the coefficients $a_{i_0, ..., i_g}$ are elements of the ring $\CC\{x\}$; 
             \item   the indices satisfy the inequalities   $0 \leq i_k \leq n_{k+1} -1$ for any 
                    $k \in \{0, ..., g-1\}$ and $i_g$ is any    non-negative integer; 
             \item $J(h)$ is the set of indices with non-vanishing coefficients, that is, 
                 the support of $h$ for this expansion. 
         \end{itemize} 
    Moreover, the intersection numbers 
    $I_0(V(f), V(a_{i_0, ..., i_{g-1},0} f_0^{i_0}   \cdots f_{g-1}^{i_{g-1}}))$ 
    of $C$ with the curves defined by the various terms of the sum 
                  which do not vanish on $C$ are  pairwise distinct. 
\end{prop}

Let us recall now some facts about the minimal embedded resolution 
$\pi_S : S \to \CC^2$ of $C=V(f)$.  Extend first the notation $f_k$ 
to the index $k=-1$ by:
  \[ f_{-1} := x. \]
Then (see \cite[Cor. 5.6]{PPP}):

\begin{prop} \label{prop:topint}
   The components of $Exc_S$ which intersect exactly one other component 
   of $Exc_S$ are precisely those which are intersected by the strict transforms 
   of the branches $V(f_k)$, for $k \in \{-1, ..., g-1\}$. Those components are pairwise 
   distinct, and if we denote by $E_C^{(k)}$ the one intersected by $V(f_k)$, then 
   $f_k$ is a curve associated with $E_C^{(k)}$, for all $k \in \{-1, ..., g-1\}$ (see Figure 
   \ref{fig:ApproxRoots2}).
\end{prop}

More generally, if a prime divisor $E$ is given and $S$ denotes its minimal model, 
then $S$ is an embedded resolution of the branch $V(h_E)$, where $h_E$ is 
any function associated with $E$. Apply the previous proposition to $f = h_E$. 
It may be shown that the components $(E_{V(h_E)}^{(k)})_{-1 \leq k \leq g-1}$ 
are independent of the choice of function $h_E$. We will denote them simply 
by $(E^{(k)})_{-1 \leq k \leq g-1}$. Extend this notation by saying that 
$E^{(g)}:=E$. Therefore, the strict transform of $V(f_k)$ on $S$ is a curvetta for $E^{(k)}$, 
for all $k \in \{-1, ..., g\}$.

\begin{figure}%
\includegraphics[width=110mm]{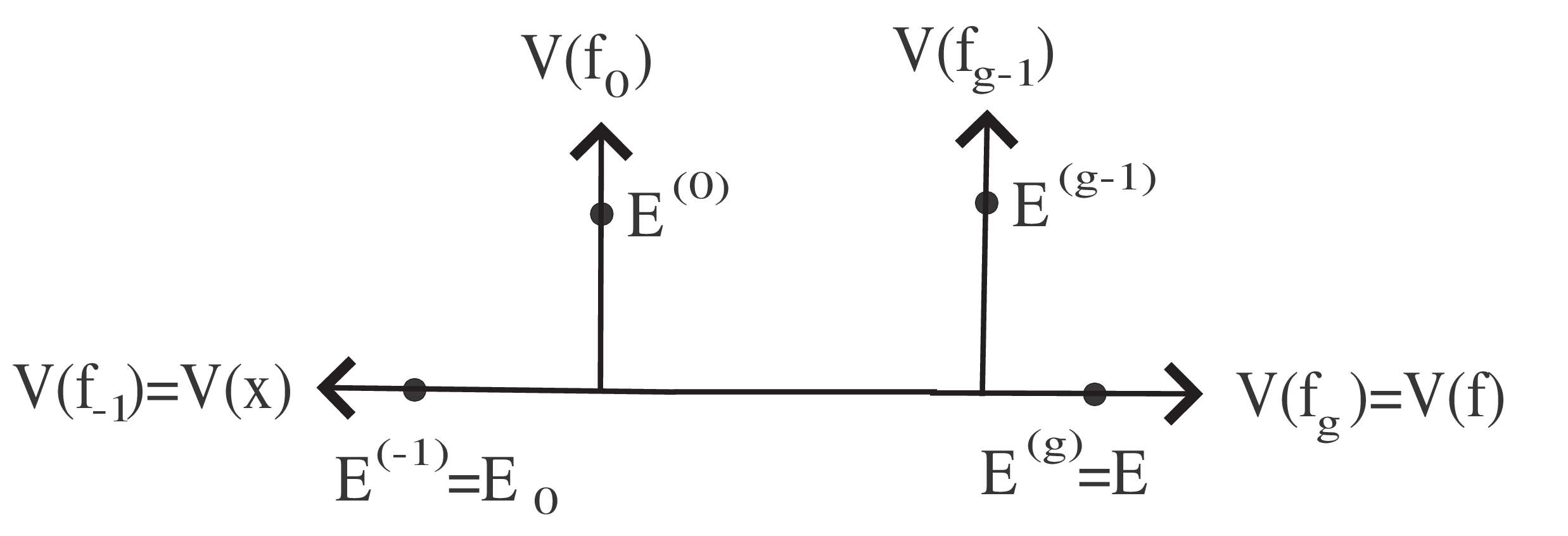}%
\caption{The strict transform of the branch $V(f_k)$ defined by the approximate root $f_k$ is represented by an arrow labeled ``$V(f_k)$''}
\label{fig:ApproxRoots2}%
\end{figure}

Using the two previous propositions, one may reprove the following description 
of the valuation $\nu_E$ in terms of the approximate roots $f_k$, which is 
also a consequence of Spivakovsky's  \cite[Thm. 8.6]{Spi}:

\begin{lem} \label{lem:val_approx}
       For any $h \in \mathcal{O}$, if one expands it as in Proposition 
       \ref{prop:maintool}, then one has:
           \begin{equation}\label{eq:eqE}  \nu_E(h) = 
           \min \{ \nu_E(a_{i_0, ..., i_g} f_0^{i_0} \cdots f_g^{i_g} ) \: \:  
                   | \: \:   (i_0, ..., i_g) \in J(h)  \}.   \end{equation}         
\end{lem}

\begin{proof}
      We reason   by contradiction, assuming that: 
            \begin{equation} \label{eq:contrad} 
                   \nu_E(h) > \eta
            \end{equation} 
         where $\eta$ denotes  the right-hand-side of (\ref{eq:eqE}). As $\nu_E$ 
         is a valuation, this implies that there are at least two terms which realize that 
         minimum, and that the sum of such terms has $\nu_E$-value strictly bigger than $\eta$. 
         Replacing $h$ by this sum, we may therefore assume that:
            \begin{equation}  \label{eq:eqall}
                 \nu_E(a_{i_0, ..., i_g} f_0^{i_0} \cdots f_g^{i_g} ) = \eta,  \: \:  \forall \: \: 
                     (i_0, ..., i_g) \in J(h). 
            \end{equation}
          Dividing $h$ by the highest power of $f = f_g$ appearing in the expansion, 
          we may assume that there are terms in it with $i_g =0$. For such terms, 
          Lemma \ref{lem:compar} implies that: 
             \begin{equation} \label{eq:twoval}
                    \eta =  \nu_E(a_{i_0, ..., i_{g-1}, 0} f_0^{i_0} \cdots f_{g-1}^{i_{g-1}} ) = 
                     I_0(V(f),V(a_{i_0, ..., i_{g-1}, 0} f_0^{i_0} \cdots f_{g-1}^{i_{g-1}})). 
             \end{equation}
         Indeed, Proposition \ref{prop:topint} shows that the strict transforms on $S$ of the two 
         curves $V(f)$ and \linebreak 
         $V(a_{i_0, ..., i_{g-1}, 0} f_0^{i_0} \cdots f_{g-1}^{i_{g-1}})$ 
         do not intersect. Proposition \ref{prop:maintool} implies that there is only one such 
         term $a_{i_0, ..., i_{g-1}, 0} f_0^{i_0} \cdots f_{g-1}^{i_{g-1}}$ in the expansion 
         of $h$, therefore:
            \begin{equation} \label{eq:twovalbis}
                     I_0(V(f), V(h)) = 
                        I_0(V(f), V(a_{i_0, ..., i_{g-1}, 0} f_0^{i_0} \cdots f_{g-1}^{i_{g-1}})). 
             \end{equation}
          Combining now, in order, formula (\ref{eq:contrad}), Lemma \ref{lem:compar}, 
          formula (\ref{eq:twovalbis}) and formula  (\ref{eq:twoval}), we get the contradiction:
            \[\eta < \nu_E(h) \leq I_0(V(f), V(h)) = I_0(V(f), V(a_{i_0, ..., i_{g-1}, 0} f_0^{i_0} 
                \cdots f_{g-1}^{i_{g-1}})) = \eta.   \]
          The lemma is proved.   
\end{proof}

We are ready to prove our improvement of Proposition 
\ref{prop:val} in the case when both divisors $E$ and $F$ are prime (recall 
from Definition \ref{def:assoc}
that $h_D$ denotes a function associated with the divisor $D$):

\begin{prop}  \label{prop:prime}
    Let $E\neq F$ be different prime divisors above $O$. Denote by $S$ 
    the minimal model of  $E$.  The following statements are equivalent: 
  \begin{enumerate}
     \item $E\leq_\nu F$; 
     \item $\nu_E(h_{E^{(k)}})\leq \nu_F(h_{E^{(k)}})$ for all the prime divisors 
     $(E^{(k)})_{-1 \leq k \leq g}$. 
   \end{enumerate}
\end{prop}

\begin{proof}
   The implication $(1)\Rightarrow(2)$ is an immediate consequence of the definition 
    of the valuative partial order. 
    
    Let us prove that $(2)\Rightarrow(1)$.  A function associated with $E$ can be chosen 
    so that it is a Weierstrass polynomial $f$. Let 
        $(f_k)_{0 \leq k \leq g}$ be its approximate roots. By Proposition 
        \ref{prop:topint}, each $f_k$ is a function associated with $E^{(k)}$. Our hypothesis 
        may therefore be rewritten as:
          \begin{equation} \label{eq:rewrite}
                 \nu_E(f_k)\leq \nu_F(f_k), \: \: \forall \: \: k \in \{-1 \leq k \leq g\}. 
          \end{equation}
      Take any function $h \in \mathcal{O}$, which is identified to 
      $\CC\{x,y\}$ by the choice of local coordinates $(x,y)$. By Weierstrass' 
      preparation theorem, it is equal to a polynomial in $\CC\{x\}[y]$ multiplied 
      by a unit. Therefore we may assume that $h \in \CC\{x\}[y]$. Expand 
      it $(f_0, ..., f_g)$-adically (see Proposition \ref{prop:maintool}):
             \[ h = \sum_{(i_0, ..., i_g) \in J(h)} a_{i_0, ..., i_g} f_0^{i_0} \cdots f_g^{i_g}. \]
       As $F$ is prime, $\nu_F : \mathcal{O} \to \NN \cup \{ + \infty \}$ is a valuation, 
       therefore:
          \begin{equation} \label{eq:ineqF}
             \nu_F(h) \geq \min \{ \nu_F(a_{i_0, ..., i_g} f_0^{i_0} \cdots f_g^{i_g} ) \: \:  
                   | \: \:   (i_0, ..., i_g) \in J(h)  \}.
          \end{equation}
      The coefficients $a_{i_0, ..., i_g}$ being moreover series in the variable $x$ alone, 
      one deduces also the following equality from the fact that $\nu_F$ is a valuation:
         \[ \nu_F(a_{i_0, ..., i_g} f_0^{i_0} \cdots f_g^{i_g} ) = 
               \ord_x( a_{i_0, ..., i_g}) \cdot  \nu_F(x) + 
                      \sum_{k=0}^g i_k \cdot \nu_F(f_k). \] 
         As $E$ is also prime, one has the analogous equality obtained by 
         replacing $F$ by $E$. 
        Therefore, our hypothesis (\ref{eq:rewrite}) implies that (recall that $f_{-1}=x$):
           \[ \nu_F(a_{i_0, ..., i_g} f_0^{i_0} \cdots f_g^{i_g} )  \geq \nu_E(a_{i_0, ..., i_g} f_0^{i_0} 
           \cdots f_g^{i_g} ), \:  \:  \forall \:  \:  (i_0, ..., i_g) \in J(h), \] 
        from which we get:
          \[\min \{ \nu_F(a_{i_0, ..., i_g} f_0^{i_0} \cdots f_g^{i_g} ) \: \:  
                   | \: \:   (i_0, ..., i_g) \in J(h)  \} \geq \min \{ \nu_E(a_{i_0, ..., i_g} f_0^{i_0} 
                      \cdots f_g^{i_g} ) \: \:    | \: \:   (i_0, ..., i_g) \in J(h)  \}. \]
         
            By combining Lemma \ref{lem:val_approx} 
            with the inequality (\ref{eq:ineqF}), we would deduce 
            that $\nu_F(h) \geq \nu_E(h)$. This inequality being true for any $h \in \mathcal{O}_0$, 
            we would get the desired conclusion $\nu_F \geq_{\nu} \nu_E$.  
\end{proof}

Let $[(U_j)_0^k;(\tilde{\beta}_j)_0^k]$ be a sequence of key polynomials (a SKP as defined in \cite{VT}) of length $k$ with $1\leq k<\infty$. By \cite[Chap. 2, Thm. 2.8]{VT}, since the sequence 
is finite, there exists a prime divisor $E$
such that it is the sequence of key polynomials associated with the valuation $\nu_E$. 
The polynomials of the sequence coincide with the approximate roots of the last polynomial of the sequence. Interpreted in this language, the previous result complements 
\cite[Chap. 2, Thm. 2.8]{VT} in the following sense:

\begin{cor}\label{cor:key}
Let $[(U_j)_0^k;(\tilde{\beta}_j)_0^k]$ be a sequence of key polynomials 
(a SKP as defined in \cite{VT}) of length $k$ with $1\leq k<\infty$ and  $\tilde{\beta}_k\in \QQ$. 
Then there exists a unique divisorial centered valuation $\nu_k$ satisfying
\begin{itemize}\item[(Q1)] $\nu_k(U_j) = \tilde{\beta}_j$ for $0 \leq j \leq k$;
\item[(Q2')] $\nu_k\leq \nu$ for any divisorial valuation $\nu$ satisfying 
$\nu(U_j) \geq \tilde{\beta}_j$ for $0 \leq j \leq k$.
\end{itemize}
\end{cor}

\subsection{The valuative criterion and deformations}\label{sec:def}

$\:$ 

\medskip

Some of the results in this section reinterpret and generalize results of  
\cite{AR} in terms of the valuative criterion.

We begin by explaining some notations used in the sequel.

\begin{notation}\label{not:def} 
      A germ of holomorphic mapping: $$H:(\CC^3,0)\to (\CC,0)$$ 
$$(x,y,s)\to H(x,y,s)$$ such that $H(s, 0, 0) = 0$ 
is said to be a \emph{deformation} of the plane curve $H(x,y,0)=0$, which is called the \emph{special curve} of the deformation. We denote $h_s(x,y)=H(s,x,y)$. It is known that there exists  
a representative $\Lambda$ of $(\CC,0)$ such that for $s\in\Lambda\setminus\{0\}$, 
the topological type of the curves $(V(h_s),0)$ is constant.  
We call them the \emph{generic curves} of the deformation. 
We denote the deformation by  $(h_s)_{s\in\Lambda}$ 
or simply by $(h_s)_{s}$. We say that $\Lambda$ is the \emph{space of parameters} 
and in this context, we will always assume that $s\in\Lambda$.   
\end{notation}

\begin{notation}\label{not:def2}
Let $f: (\CC^2, 0) \to (\CC, 0)$ be a holomorphic function germ 
and let 
$S$ be a model of $(\CC^2,0)$. Let $Exc_S=\bigcup_iD_i$ 
be the decomposition into irreducible components of its exceptional set. 
We say that $H_f=\sum c_iD_i$ is the \emph{divisor associated} with $f$ in the model 
$S$ if $c_i= \widetilde{V(f)}\cdot D_i$ for all $i$, where $\widetilde{V(f)}$ 
denotes the strict transform of $V(f)$ on $S$.

In particular, given a deformation $(g_s)_s$, we can define $E:=H_{g_s}$ for $s\neq 0$ small enough, since it does 
not depend on this choice. 
    
\end{notation}
As a first result relating deformations and the valuative criterion, we have the following:

\begin{prop}\label{lem:val_def1} 
     Let $G=(g_s)_s$ be a deformation. Let $S$ be any model and let $F:=H_{g_0}$ and 
     $E:=H_{g_s}$ (with $s \neq 0$) as in Notation \ref{not:def2}.  
     Then $\nu_E\leq\nu_F$.
\end{prop}

\begin{proof} 
We proceed as in \cite[Section 3.2]{Nash_surfaces}. We use Notation  \ref{not:def2}.

Consider the mapping: 
    $$\sigma:=\pi\ \times\ Id_\Lambda:S\times\Lambda\to \CC^2\times \Lambda$$ 
 where $(\Lambda,0)\contneq (\CC,0)$ is the space of parameters $s$ of the deformation.  

Let $Y$ be the strict transform of $V(G)$ in $S\times \Lambda$ by $\sigma$. 
Denote $Y_s:=Y\cap(S\times\{s\})$. Observe that: 
     $$Y_s=\widetilde{V(g_s)}\ \ for\  s\neq 0$$ 
    $$Y_0=\widetilde{V(g_0)}+\sum_kd_kD_k, \ \ \ with \ \ d_k\geq 0.$$ 
Due to the invariance of the intersection number under deformation, we have the system of
equations: 
    \begin{equation}\label{eq:4}Y_0\cdot D_i=Y_s\cdot D_i\end{equation} 
for any $i$. 
If we define $b_i:=\widetilde{V(g_0)} \cdot D_i$ and $a_i:=Y_s \cdot D_i$, 
then the system (\ref{eq:4}) can be rewritten as follows using the intersection matrix 
$M=(D_i\cdot D_j)_{i,j}$: 
\begin{equation}
       (b_0,b_1,...,b_n)^t+M(d_1,..,d_n)^t=(a_0,...,a_n)^t
\end{equation}
or equivalently:
\begin{equation}\label{eq:8}
        -M^{-1}(b_0-a_0,...,b_n-a_n)^t=(d_1,...,d_n)^t\geq 0.
\end{equation}
The entries of the matrix $-M^{-1}$ are exactly 
$\nu_{D_{i}}(h_{D_j})=I_0(V(h_{D_i}),V(h_{D_j}))$ (see Lemma \ref{lem:compar}), 
and since $d_i\geq 0$ for all $i$, the inequality (\ref{eq:8}) says exactly  
$\nu_{E}(h_{D_i})\leq \nu_{F}(h_{D_i})$ for all $i$. Then, using the implication
$(4) \Rightarrow (1)$  of Proposition \ref{prop:val}, we conclude that $\nu_{E}\leq \nu_{F}$ .  
\end{proof}

Proposition \ref{prop:val_def2} below is a kind of reciprocal of the previous one. 
In particular, it shows that with the hypothesis of the valuative inequality we can control 
the generic member of a pencil of curves. For the case of prime divisors and 
$V(\widetilde{g})$ a curvetta for $E$, this was proved in \cite{AR}. 

To state Proposition \ref{prop:val_def2} in the maximal generality,  we need the following 
definition, which weakens Definition \ref{def:assoc}, in the sense that one does not impose 
now any transversality hypothesis:

\begin{defn} 
   Let $F=\sum_ib_iD_i$ be a divisor in a model $S$ of $(\CC^2,0)$ with 
   $Exc_S=\bigcup_iD_i$. We say that a function $f$ is \emph{weakly associated with 
   $F$ in $S$} if $b_i = \nu_{D_i}(f)= \widetilde{V(f)}\cdot D_i$ for all $i$.
\end{defn}

We need also the following lemma: 

\begin{lem}\label{lem:weak1} 
     Let $h$ be a function weakly associated with a divisor $F= \sum b_i D_i$ 
     in a model $S$. Then $\nu_{D_i}(h)$ does not depend on the choice of $h$, 
     for any prime component $D_i$ of $Exc_S$. 
 
      As a consequence, let $S$ be obtained by blowing up points $\{x_i\}_{i\in I}$ and 
      $\tilde{h}_i$   denote the local equation of the strict transform of $h$ at $x_i$. If $h$ 
      is weakly associated with $F$, then $m_{x_i}(V(\tilde{h}_i))$ does not depend on 
      the choice of $h$ for all $i\in I$. 
\end{lem}

\begin{proof} 
Let $h$ be a holomorphic function weakly associated with $F$ in $S$.
We know that $\nu_D(h)=I_0(h, h_D)$ for some associated function $h_D$ to $D$ whose strict transform in $S$ does not meet the strict transform of $h$. 
But we have the equality $I_0(h,h_D)=\pi^*(h_D)\cdot \tilde{h}$, which only depends on the intersection products $b_i:=\widetilde{V(h)}\cdot D_i$. 
 
    Let us prove the consequence. Let $\{D_j\}_{j\in J}$ be the prime divisors meeting $x_i$ 
   at the minimal model where $x_i$ appears. We just observe the equality:
   $$m_{x_i}(V(\widetilde{h}_i))=\nu_{D_i}(h)-\sum_{j\in J}\nu_{D_j}(h).$$ 
\end{proof}

\begin{prop}\label{prop:val_def2}
    Let $E=\sum_ia_iD_i$, $F=\sum_ib_iD_i$ be divisors such that $\nu_E\leq \nu_F$. 
    Let $V(g)$ and $V(f)$ be curves weakly associated with $E$ and $F$ in the minimal 
    model $S$ of $E+F$. 
    Then $V(f+sg)$ are weakly associated with $E$ in $S$ for all $s\neq 0$ small enough. 

    Moreover, if $g$ is associated with $E$ then so is $V(f+sg)$ for $s\neq 0$ small enough.  
\end{prop}

\begin{proof} 

Let $\{x_i\}_{i\in I}$ be the infinitely near points that we blow up in order to obtain the minimal model of $E$. Here $I$ is partially ordered by saying that $i <j$ precisely when $x_j$ is 
infinitely near $x_i$. We denote $x_0: = 0$ and for any $i >0$ we will denote by 
$x_{i-1}$ the point which is blown up in order to create the minimal model of $x_i$.
Let $E_i$ be the divisor obtained after blowing up $x_i$.  Let $f_i$ be a local 
defining function at $x_i$ of the 
total transform of $f$ and $\widetilde{f}_i$ a local defining function at $x_i$ 
of the strict transform 
of $f$. We use the same notations for $g$ and for $f + sg$. 
Let $\{m_i\}_{i\in I}$ be the multiplicity 
sequence of $g$, that is, $m_i=m_{x_i}(\widetilde{g}_i)$.

Note that by the choice of $f$ and $g$ and the inequality $\nu_E\leq \nu_F$ we have: 
\begin{equation}
\label{ineq-nueva}
    \nu_{E_i}(g)=I_O(h_{E_i},g)=I_O(h_{E_i},h_E)=
      \nu_E(h_{E_i})\leq\nu_F(h_{E_i})=I_O(h_{E_i},h_F)=I_O(h_{E_i},f)=\nu_{E_i}(f)
\end{equation}
for any $E_i$ in $Exc_S$  and for $h_D$ denoting functions associated with the divisor $D$, 
chosen generically relative to $f$ and $g$. 

1. We prove that for $s\neq 0$ the strict transform of $V(f+sg)$ after the 
blow-up of $x_{i-1}$ passes  through $x_i$ and has the same multiplicity as $V(g)$, 
that is, $m_{x_i}(\widetilde{(f+sg)}_i)=m_i$.
 
We proceed by induction on the partial order of $I$.

Initial step at $x_0$. Both $f$ and $g$ go through the origin $x_0$ and $m_{x_0}(f+sg)=m_{x_0}(g)=m_0$ 
for $s\neq 0$ small enough because 
$m_{x_0}(f)\geq m_{x_0}(g)$: the multiplicity of $f+sg$ is different from that of $g$  only  for certain $s$, so it is equal to it if $s\neq 0$ is small enough.  

Let's do in detail the case where $x_i$ is a satellite point and $x_i=E_{i-1}\cap E_j$. 
Assume that $E_{i-1}=V(y)$ and $E_j=V(x)$ for local coordinates $(x,y)$ around $x_i$. 
Then we may write: 
     $$f_i=x^{\nu_{E_j}(f)}y^{\nu_{E_{i-1}}(f)}\widetilde{f}_i,$$
$$g_i=x^{\nu_{E_j}(g)}y^{\nu_{E_{i-1}}(g)}\widetilde{g}_i$$
and so the total transform of $f+sg$ around $x_i$ is:   
$$(f+sg)_i=x^{\nu_{E_j}(f)}y^{\nu_{E_{i-1}}(f)}\widetilde{f}_i+
        s\cdot x^{\nu_{E_j}(g)}y^{\nu_{E_{i-1}}(g)}\widetilde{g}_i=$$
$$=x^{\nu_{E_j}(g)}y^{\nu_{E_{i-1}}(g)}(x^{\nu_{E_j}(f)-\nu_{E_j}(g)}y^{\nu_{E_{i-1}}(f)-
        \nu_{E_{i-1}}(g)}\widetilde{f}_i+s\cdot\widetilde{g}_i)$$
where what is inside the parenthesis is $\widetilde{(f+sg)}_i$. 
If the two strict transforms $V(\widetilde{f}_i)$ and $V(\widetilde{g}_i)$ pass through $x_i$, 
then it is also the case for $V(\widetilde{(f+sg)}_i)$.  If only $V(\widetilde{g}_i)$ 
passes through $x_i$, in order to have that $\widetilde{(f+sg)}_i$ is zero in $x_i$, 
we need to guarantee that $\nu_{E_{i-1}}(f)-\nu_{E_{i-1}}(g)$ and $\nu_{E_{j}}(f)-\nu_{E_{j}}(g)$ 
are not both zero. But if they were both zero, then $\nu_{E_{i-1}}(f)+\nu_{E_{j}}(f)=
      \nu_{E_{i-1}}(g)+\nu_{E_{j}}(g)$ and since $m_i>0$  we would have that:  
$$\nu_{E_{i}}(f)=\nu_{E_{i-1}}(f)+\nu_{E_{j}}(f)<\nu_{E_{i-1}}(g)+\nu_{E_{j}}(g)+m_i=\nu_{E_{i}}(g)$$ 
which cannot be the case by inequality~(\ref{ineq-nueva}). 

Let us show now that the multiplicity of $\widetilde{(f+sg)}_i$ at $x_i$ is 
$m_i=m_{x_i}(\widetilde{g}_i)$. One has
$m_{x_i}\tilde{f}_i=\nu_{E_i}(f)-\nu_{E_{i-1}}(f)-\nu_{E_j}(f)$ and $m_{x_i}\tilde{g}_i=\nu_{E_i}(g)-\nu_{E_{i-1}}(g)-\nu_{E_j}(g)$, therefore the inequality: 
\begin{equation}
\label{eq:m}
       m_{x_i}(x^{\nu_{E_j}(f)-\nu_{E_j}(g)}y^{\nu_{E_{i-1}}(f)-
           \nu_{E_{i-1}}(g)}\widetilde{f}_i)\geq m_{x_i}(\widetilde{g}_i)
\end{equation} 
is equivalent to the inequality $\nu_{E_{i}}(f)\leq \nu_{E_i}(g)$, which we have 
proved before. Moreover,  the equality in (\ref{eq:m}) (and the possible jump of multiplicity 
of $\widetilde{(f+sg)_i}$) 
may only occur for certain isolated values of $s$ 
(and only when we have the equality $\nu_{E_{i}}(f)=\nu_{E_i}(g)$), 
therefore not for $s\neq 0$ small enough.  

The case where $x_i$ is a free point is analogous. 

2. We claim that the fact that $\widetilde{g_i}$ and $\widetilde{(f+sg)_i}$ have the same multiplicity at $x_i$ for any infinitely near point $x_i$ which is blowing up center 
for the model $S$ implies the equality: 
$$\widetilde{V(f+sg)}\centerdot D_i=\widetilde{V(g)}\centerdot D_i$$
for any exceptional divisor $D_i$ in $S$ and $s\neq 0$ small enough. This exactly means that $f+sg$ is weakly associated with $E$ for $s\neq 0$ small enough. 

Let us prove this claim. 
Give a total order to the infinitely near points involved in the blowing up process 
leading to $S$, which is compatible with the natural partial order given by domination.
Let $x_i$ be the blowing up center giving rise to $D_i$ in its minimal model. 
Denote by $\widetilde{\widetilde{V(g)}}_i$
the strict transform of $V(g)$ at the result of blowing up the point $x_i$. 
In a neighbourhood of $D_i$ it coincides with the strict transform of
$V(\widetilde{g_i})$ by the blowing up of $x_i$. Therefore we have the equality: 
     $$\widetilde{\widetilde{V(g)}}_i\centerdot D_i=m_{x_i}(\widetilde{g_i})=
           m_{x_i}(\widetilde{(f+sg)_i})=\widetilde{\widetilde{V(f+sg)}}_i\centerdot D_i.$$

Let $x_j$ be the blowing up center appearing right after $x_i$. If $x_j$ is not located 
in $D_i$, then we have the equalities:
    $$\widetilde{\widetilde{V(g)}}_j\centerdot D_i=\widetilde{\widetilde{V(g)}}_i\centerdot D_i$$
and:
    $$\widetilde{\widetilde{V(f+sg)}}_j\centerdot D_i=
        \widetilde{\widetilde{V(f+sg)}}_i\centerdot D_i.$$

If $x_j$ is on $D_i$, then we have the equalities:
     $$\widetilde{\widetilde{V(g)}}_j\centerdot (D_i+D_j)=
           \widetilde{\widetilde{V(g)}}_i\centerdot D_i=m_{x_i}(\widetilde{g_i}),$$
    $$\widetilde{\widetilde{V(f+sg)}}_j\centerdot (D_i+D_j)=
          \widetilde{\widetilde{V(f+sg)}}_i\centerdot D_i=m_{x_i}(\widetilde{(f+sg)_i}),$$ 
     $$\widetilde{\widetilde{V(g)}}_j\centerdot D_j=m_{x_j}(\widetilde{g_j}),$$ 
     $$\widetilde{\widetilde{V((f+sg)}}_j\centerdot D_j=m_{x_j}(\widetilde{(f+sg)_j}).$$ 
Since we have shown the equality 
     $m_{x_i}(\widetilde{g_i})=m_{x_j}(\widetilde{(f+sg)_i})$ for any $i$,  
     we deduce the equality:
  $$\widetilde{\widetilde{V(g)}}_j\centerdot D_i=\widetilde{\widetilde{V((f+sg)}}_j\centerdot D_i.$$

Iterating this procedure for all blowing up points, the claim is proven.

3. Let us show now the last assertion of the proposition. Consider the pencil 
$\lambda f+\mu g$ for $(\lambda:\mu)\in\mathbb{P}^1$. The fact that $g$ is 
associated with $E$ and 
$f+sg$ is weakly associated with $E$ for almost all $s$ implies that the generic member of the pencil is associated with $E$: the intersection number 
$\widetilde{V(f+sg)}\centerdot D_i$ equals $b_i$ and there is a member of the pencil (the one corresponding to the
parameter value $(0:1)$, which is $g$) for which the set $\widetilde{V(f+sg)}\cap D_i$ consists of $b_i$ distinct points. Since this is the most generic behaviour we have that for generic $s$ 
the set $\widetilde{V(f+sg)}\cap D_i$ consists of $b_i$ distinct points.   
\end{proof}

Putting the last two propositions together we have the following theorem:   

\begin{theo}\label{theo:val_def}
      Let $E$ and $F$ be divisors over the origin of $\CC^2$ and let $S$ be the minimal model 
      containing the divisor  $E+F$. The following are equivalent: 
             \begin{enumerate}[(a)]
			\item  $\nu_E\leq \nu_F$; 
			\item there exists a deformation $G=(g_s)_s$ with $g_0$ weakly associated 
			      with $F$ in $S$ and $g_s$ weakly associated with $E$ in $S$, 
			       for $s\neq 0$  small enough; 
			\item there exists a deformation $G=(g_s)_s$ with $g_0$ associated with $F$ 
			       in $S$ and $g_s$ associated with $E$ in $S$, for $s\neq 0$ 
                                   small enough;  
                            \item there exists a linear deformation $(g_0+sg)_s$ with $g_0$ 
                                  associated with $F$ in $S$ and $g$ and $g_s$ associated with $E$ 
                                  in $S$, for $s\neq 0$ small enough.  
              \end{enumerate}
 \end{theo}
		
For the case that $F$ is prime, we have the following finer theorem, of which
Theorem~\ref{theo:main1} stated in the introduction is the special case where 
$E$ is also prime:
		
\begin{theo}\label{theo:val_def_prim}
             Let $F$ be a prime divisor, and $E$ be any divisor. 
             Then the following are equivalent: 
\begin{enumerate}[(a)]
     \item $\nu_E\leq\nu_F$;  
       \item there exists a deformation $G=(g_s)_s$  with $g_0$ associated with $F$ 
            in the minimal model of $F$ and $g_s$ associated with $E$ in the minimal model 
            of $E$ for $s\neq 0$ small enough.
\end{enumerate}
\end{theo}
\begin{proof}
The implication $(a)\Rightarrow(b)$ follows from $(a)\Rightarrow(d)$ of Theorem \ref{theo:val_def}.

   It only remains to prove that $(b)\Rightarrow(a)$. 
Assume that we have a deformation as in the statement. 
Consider the minimal model $S$ for $E$. We claim that the divisor $F$ does not appear in $S$ unless we have the equality $E=F$,
case in which the implication is obvious. 

 Let $E = \sum a_iE_i$. If $F$ appears in the minimal model of $E$ then there exists $E_i$ in the 
support of $E$ such that either it is equal to $F$ or it dominates it. But then, 
by Remark~\ref{rem:dom-val},  we have the
valuative domination $\nu_{E_i}\geq \nu_F$. Since by hypothesis $\nu_E\leq \nu_F$ 
and obviously $\nu_{E_i}\leq\nu_E$,
we have $\nu_{E_i}=\nu_F$. We deduce the equality $E_i=E=F$. 

Now, we keep blowing up until we get a model where  $F$ also appears. It may happen that the strict transform of 
$V(g_s)$ for small enough $s\neq 0$ meet now another 
divisor $E'>E$. We can apply now Lemma \ref{lem:val_def1} to obtain $E'\leq_\nu F$. Since we have that $E\leq_\nu E'$ we get $E\leq_\nu F$ as desired.
\end{proof}

\begin{rem}\label{re:mas} 
       Given a deformation as in (b) in Theorem \ref{theo:val_def_prim} 
     we can consider the minimal model $S$ of $E+F$ and
    the divisors $H_{g_0}$ and $H_{g_s}$ as in Notation \ref{not:def2}. A quick reflection 
    will convince the reader that the equalities $F=H_{g_0}$ or $E=H_{g_s}$ do not necessarily hold,  
  for example if $g_0$ is also associated with a prime divisor $F'>_dF$ which appears 
  in the minimal model of $E$. 
\end{rem}

\subsection{Classical adjacency of functions and deformations fixing the free points}
\label{sec:def_fix}
$\:$ 
\medskip

The classical \emph{adjacency problem for plane curves} asks whether, 
 given  any two different embedded topological types encoded, for instance, by the two dual graphs $\calG_0$ and $\calG_1$ of the minimal embedded resolution of two curves germs, 
 there exists a deformation $(h_s)_s$ where $V(h_0)$ and $V(h_{s\neq 0})$ have the embedded topological types given by $\calG_0$ and $\calG_1$ respectively. 
We say in this case that the topological type given by  $\calG_0$ is \emph{upper-adjacent} to the one given by $\calG_1$. 

Let $(h_s)_s$ be a curve deformation and $C$ any plane curve germ at
the origin. It is classical that:
\begin{equation}\label{eq:sem_int}I_0(V(h_0),C)\geq I_0(V(h_s),C),\end{equation}
for $s \neq 0$ small enough. 

\begin{rem}
    Gorsky and N\'emethi proved in \cite[Proposition 4.5.1]{NG} a refined criterion, 
    consisting on an  upper semi-continuity of the Hilbert function associated with a 
    (multi)-branch, which obstructs the existence of Arnold adjacencies in the case 
    in which the special curve is irreducible. Their proof boils down to the inequality 
    (\ref{eq:sem_int}).
\end{rem}

Given a deformation as in Notation \ref{not:def},  we know that the family of curves $V(h_s)$ 
for $s\neq 0$ have the same embedded topological type, that is, the final divisors of 
their  minimal embedded resolutions have the same combinatorial type. These divisors 
can change with $s\neq 0$ because the free infinitely near points appearing as blowing 
up centers in the minimal embedded resolutions of the curves $V(h_{s\neq 0})$ 
can move continuously with $s$ without changing the combinatorial type. 
In the deformations in item (c) of Theorem \ref{theo:val_def} this does not happen and 
the generic curves for small value of the parameter $s\neq 0$ can be resolved by a 
single sequence of point blow ups, and can sometimes even lift together after further blowups. 
We codify this phenomenon in a definition (in \cite{AR}, 
this corresponds to a deformation associated with a cluster of infinitely near points): 

\begin{defn}\label{def:def_fix} 
    A deformation $(h_s)_{s}$ is said to 
    \emph{fix the free points for the divisor $E$} 
    if, for $s\neq 0$ small enough, the function $h_s$ is associated with $E$ 
    in the minimal model where $E$ appears. 

       A deformation $(h_s)_{s}$ is said to
        \emph{fix the free points}, if there exists a divisor $E$ such that it fixes the free points for $E$. 
        
      A deformation $(h_s)_s$ is called \emph{dicritical} along a prime divisor $E$ 
          if the strict transforms of $V(h_s)$ for $s \neq 0$ small enough move along $E$. 
          Otherwise we call it \emph{non-dicritical} along $E$.
 \end{defn}
 
\begin{rem}\label{re:def1}
\begin{enumerate}
\item Let be given a deformation $(h_s)_s$ fixing the free points for $E$. If the family 
     of strict transforms of $V(h_s)$ for all $s\neq 0$ meets some component $E_i$ of $E$  
     in a non-dicritical way, then we 
     can keep blowing up until each of the branches of the family of strict transforms 
     meets a component of the exceptional divisor of the new model in a dicritical way. 
\item The divisor for which the family is dicritical along every components is the maximal 
     divisor with respect to domination for which the deformation fixes the free points.
\end{enumerate}
\end{rem}

Observe that two curves have the same topological type if and only if they are respectively associated with two divisors $D$ and $D'$ which are combinatorially equivalent. 

\begin{rem}\label{rem:top}
  Let $D=\sum_{i}E_i$ and $D'=\sum_{i}E'_i$ be two divisors satisfying the condition  
  that $E'_i$ dominates $E_i$ in such a way that it appears after further blowing ups only in free points. Then the curve associated with $D'$ has the same topological type 
  as the curve associated with $D$.
 \end{rem}

The previous remark and Theorem~\ref{theo:val_def}  give a necessary and sufficient 
criterion for the existence of deformations fixing the free points with prescribed embedded topological types. 
Let us explain this criterion.

A topological type is determined by a class of combinatorially equivalent divisors $[E]$, 
but due to the previous remark there are further divisors whose associated curves have the 
same topological type. Denote by $[E]_{top}$ the collection of divisors whose associated 
curves have the given topological type.

Consider another topological type $[F]$ and let $[F]_{top}$ be the associated set of divisors.

By Theorem \ref{theo:val_def}, there is an adjacency fixing the free points between the 
two topological types, if and only if there exist divisors $E'\in [E]_{top}$
and $F'\in [F]_{top}$ which appear in a common model, such that we have the inequality: 
$$\nu_{E'}\leq \nu_{F'}.$$

Then, given $[F]_{top}$,  our interest is to identify all the topological types $[E]_{top}$ 
such that the previous inequality is satisfied for some $E'\in [E]_{top}$ and $F'\in [F]_{top}$.

Let us call \emph{the complexity of $E$} the length of the longest totally ordered chain of 
infinitely near points appearing as blowing up centers leading to the minimal model of $E$.

The validity of the previous inequality only depends on the combinatorial 
type of the pair $(E',F')$. The next lemma implies that, in order to check the 
existence of the adjacency fixing the free points between the topological types, 
we only have to look at pairs $(E',F')$ with $E'$ and $F'$ of bounded complexity,
and therefore at finitely many combinatorial types of the pair $(E',F')$. Being 
finite in number they can be enumerated, and this gives a method to check 
the existence of the adjacency in finite time.

Given a topological type,  there is a combinatorial class of divisors $[E]$ of minimal complexity included in $[E]_{top}$. It corresponds to the divisors associated with the minimal embedded resolutions of the curves with the given topological type. Observe that any other divisor in 
$[E]_{top}$ dominates one in this class by a blowing up process involving only free points. 
The \emph{complexity of a topological type} is, by definition, the complexity of any divisor 
associated to the minimal resolution of a curve with the given topological type.

\begin{lem}
    Denote by $M$ the maximum of the complexities of the two topological types 
    we are considering. If there exist $E'\in [E]_{top}$ and $F'\in [F]_{top}$ 
     appearing in a common model such that the inequality 
     $\nu_{E'}\leq \nu_{F'}$ holds, then there exist $E''\in [E]_{top}$ and $F''\in [F]_{top}$ such 
     that the inequality $\nu_{E''}\leq \nu_{F''}$ holds and the complexities of both $E''$ and $F''$ 
     are bounded by $M$. 
\end{lem}

\begin{proof}
   Let $S$ be the minimal model in which $E'$ and $F'$ appear. 
   By Theorem \ref{theo:val_def}, there exists a deformation $(f_s)_s$ such that the divisor 
   associated with $f_0$ in $S$ equals $F'$ and the divisor associated with 
   $f_s$ ($s\neq 0$) in $S$ equals $E'$.

       Let $\{p_1,...,p_r\}$ be a maximal linearly ordered chain of infinitely near points 
       appearing in the blowing up process leading to $S$. Let $S'$ be the model obtained 
       after blowing up the union of $\{p_1,...,p_{min\{M,r\}}\}$ over all maximal linearly 
       ordered chains leading to $S$.

       Let $F''$ be the divisor weakly associated with $f_0$ in $S'$. Since $M$ is at least 
       the complexity of the topological type given by $F'$, we have that $F''$ belongs to 
       $[F]_{top}$ (and that $f_0$ is associated to $F''$).

        Let $E''$ be the divisor weakly associated with $f_{s \neq 0}$ in $S''$; for the same 
        reason we have $E''\in [E]_{top}$ (and that $f_{s \neq 0}$ is associated to $E''$).

          Since $E''$ and $F''$ are in a common model $S'$,  Theorem \ref{theo:val_def} 
          implies the desired inequality.
\end{proof}

Moreover, we have the following reductions which help avoiding unnecessary computations: 
    \begin{itemize}
         \item using Corollary \ref{cor:max} we are reduced to check some ``maximal''
                contact orders; 
          \item when $F$ is a prime divisor, Theorem \ref{theo:val_def_prim} 
             allows us to avoid checking for the divisors $F'$ of a combinatorial type different from 
             that of $F$.
    \end{itemize}

Using the previous procedure, let us study now the adjacencies which can be realized 
fixing the free points among simple, unimodal and bimodal 
singularities with Milnor number less or equal 16, that is, those classified in \cite{Arn}.

In Figure \ref{fig:adj_simples_parab} one can see all adjacencies between simple and parabolic plane curve singularities.  The complete list of adjacencies was originally given in \cite{Saito}. 
Here an arrow directed from a singularity type $X$ to a singularity type $Y$ means that 
$X$ is \emph{upper adjacent} to $Y$, that is, there is a deformation of a singularity of type 
$X$ with generic singularities of type $Y$. \textbf{A dotted arrow means that the adjacency 
is not realizable by deformations fixing the free points.}

Among the 93 (classical) adjacencies between simple singularities with Milnor number less or equal 8, only 7 are not realized by deformations fixing the free points. 
Moreover, between $A_{*}$ and $D_{*}$ type simple singularities, 
only the adjacencies from $D_{2n+1}$ to $A_{2n}$ are not realizable in that way for any $n$.

\begin{ex} Consider the deformation: 
$$y^3+x^4 +s^2y^2+2s x^2y$$ of the singularity $E_6$. For $s\neq 0$ 
we get the singularity $A_4$. It is easy to check that the valuative inequality is 
not satisfied for the corresponding 
divisors for any contact order between them, so it is not possible to realize 
the adjacency with a deformation fixing the free points. 
\end{ex}

All adjacencies between unimodal singularities\footnote{For the convenience of the reader, we recall the different names of the unimodal singularities. The \emph{parabolic} 
ones are $X_9=T_{2,4,4}=\tilde{E}_7=X_{1,0}$ and $J_{10}=T_{2,3,6}=\tilde{E}_8=J_{2,0}$. 
For the \emph{hyperbolic} singularities we use Arnold's notation 
(not the one in \cite{Bri}) with the following translation: 
$J_{2,i}=T_{2,3,6+i}$.   We have also $X_{1,p}=T_{2,4,4+p}$ and    
$Y^1_{r,s}=T_{2,4+r,4+s}$. We use also Arnold's notations for the exceptional singularities. 
For the correspondence with the other name $S_{a,b,c}$, see \cite{Bri}.} were described in 
\cite{Bri}. All adjacencies between hyperbolic singularities $T_{a,b,c}$ are realizable by linear 
adjacencies as it was pointed out in \cite{Bri} (in this case $T_{a,b,c}\longrightarrow 
T_{\alpha,\beta,\gamma}$ if and only if $\alpha\leq a$, $\beta\leq b$, $\gamma\leq c$). 
On the contrary, not all adjacencies from exceptional to hyperbolic singularities are realizable fixing the free points as it is shown in figure \ref{fig:Arn_adj} and \ref{fig:Adj_uni}.  
In these two diagrams one can find all the adjacencies between unimodal singularities realised fixing the free points apart from the ones between hyperbolic singularities. The adjacency  from $X_{1,1}$ to $E_8$ seems not to appear in the literature  but it follows from the comparison of the Newton diagrams of their normal equations, criterion used in \cite{Arn}.

The adjacencies between bimodal singularities, between unimodal singularities and simple and between unimodal and bimodal are mainly described in Arnold where it is stated that the list is not complete. The list given in \cite{BBB} of adjacencies between bimodal singularities with Milnor number less or equal than $16$ is assumed to be more complete than the one in \cite{Arn}. We find that again almost all of these adjacencies can be realized by deformations fixing the free poitns. Those that are not, appear with dotted lines in Figure \ref{fig:Arn_adj}. The adjacencies $Z_{17}\to E_{14}$, $W_{17}\to Z_{13}$, $E_{19}\to J_{3,2}$ and $E_{20}\to J_{3,2}$ are in \cite{BBB} but not in Arnold's list. The one  $Z_{11}\to E_8$ does not appear even in \cite{BBB} but it follows also from the comparison of the Newton diagrams of their normal equations, a criterium used in \cite{Arn}. They are all realizable by families fixing the free points.

Observe that the singularity $E_8$ is simple, $Z_{11}=S_{2,4,5}$, $Z_{13}=S_{2,4,7}$, 
$E_{14}=S_{2,3,9}$ and $J_{3,2}$ are unimodal and $Z_{17}$, $E_{19}$, $W_{17}$ and $Z_{17}$ are bimodal.

\begin{figure}%
\includegraphics[width=110mm]{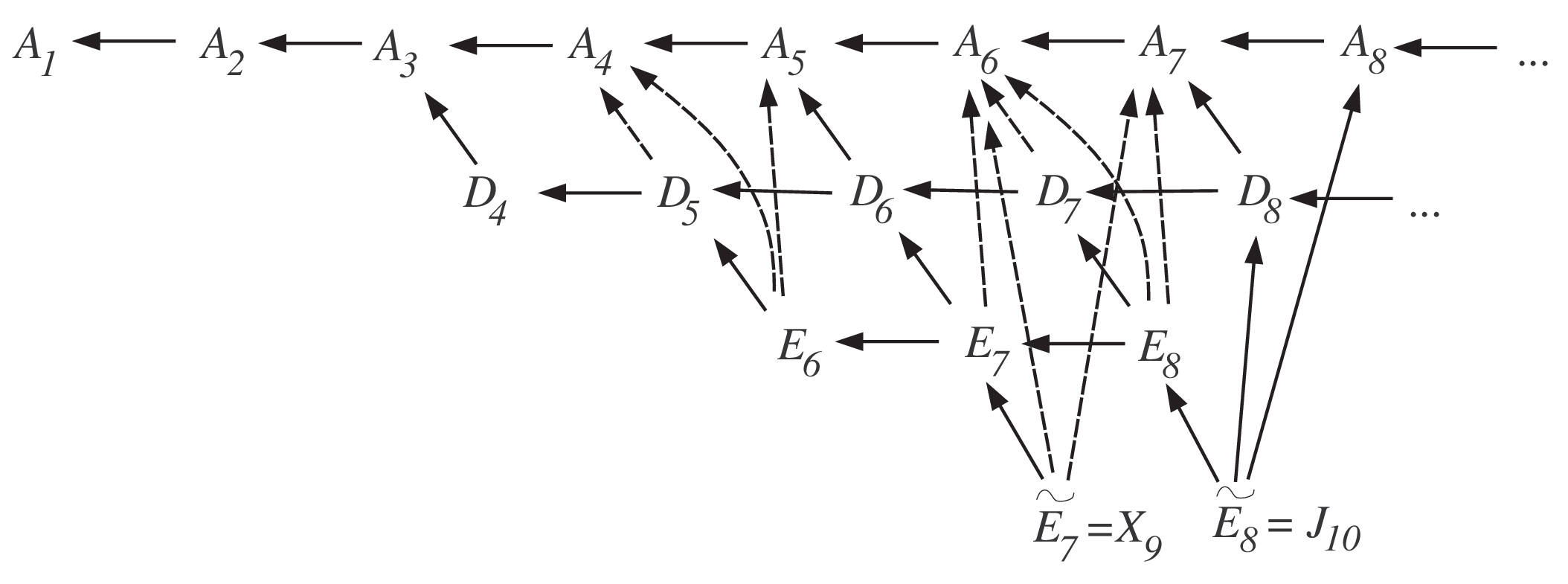}
\caption{Adjacencies realizable by deformations fixing the free points among simple 
     and parabolic singularities. Any adjacency of this type follows by transitivity from the 
     arrows in the picture.  Dotted arrows indicate the other classical adjacencies.}
\label{fig:adj_simples_parab}%
\end{figure}

\begin{figure}%
\includegraphics[width=65mm]{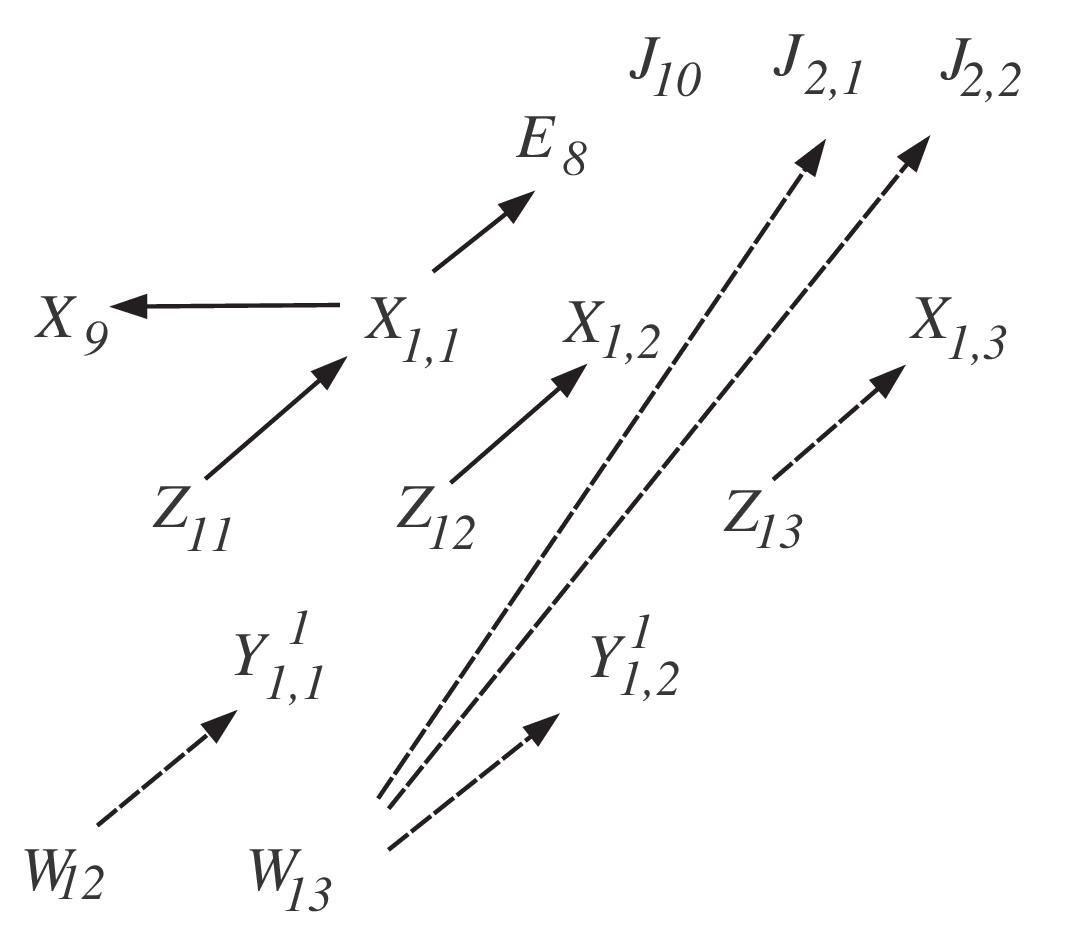}%
\caption{Adjacencies realizable by deformations fixing the free points 
    among unimodal singularities, from exceptional singularities to hyperbolic singularities. 
      Any adjacency of this type between the singularities in the diagram follows by transitivity 
      from the arrows in the picture. Dotted arrows indicate the other classical adjacencies.}%
\label{fig:Adj_uni}%
\end{figure}

\begin{figure}%
\includegraphics[width=125mm]{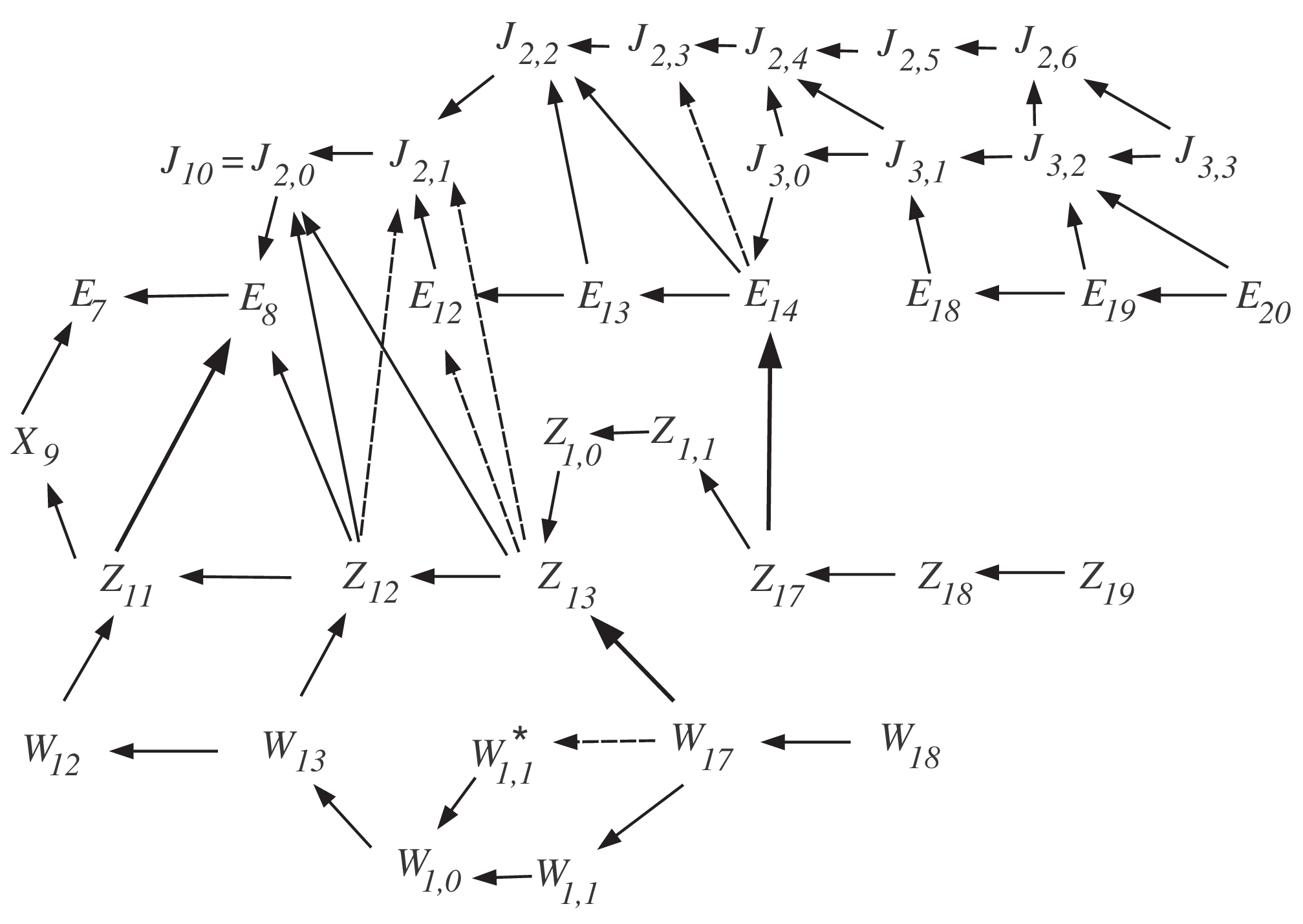}%
\caption{Adjacencies realizable by deformations fixing the free points among some simple, 
 unimodal and bimodal plane curve singularities. Any adjacency of this type between the 
 singularities in the diagram follows by transitivity from the arrows in the picture.  
 Dotted arrows indicate the other classical adjacencies.}%
\label{fig:Arn_adj}%
\end{figure}

To check these diagrams, we have considered the equations for the singularities of plane curves given in \cite{Arn} and we computed their embedded resolutions and multiplicity sequences 
using Singular (these singularities are resolved in at most 5 blow ups). 

\begin{rem}\label{rem:alg}
It is easy to build up an algorithm which, given a prime divisor $F$, finds all those combinatorial 
types of divisors $E$ such that $\nu_{E'}\leq \nu_F$ for some $E'$ with the same combinatorics 
as $E$ and with certain contact order $Cont(E',F)$. By Corollary \ref{cor:finite_val} there are 
only finitely many of them. In particular, this list of divisors would encounter all the 
embedded topological types $\calG_E$ that are adjacent to $\calG_F$ as in 
Definition \ref{def:def_fix}, by deformations fixing the free points. The second author is 
planning to program it in Singular.

First, one can proceed by looking for all the combinatorial types of prime divisors $E$   
recursively on the number of points one has to blow up to obtain their minimal model. 
The equivalence $(1)\Leftrightarrow(6)$ from Proposition \ref{prop:val} 
allows to do the computations in every step of the recursion in a fast way. 
Then, one can proceed recursively in the number of components for the non-prime case. 
Note also that according to Corollaries \ref{cor:val_max_cont} and \ref{cor:max}, in order 
to check if a divisor is adjacent to $F$ it is sufficient to check the cases of ``maximal''
contact order.

The case where $F$ is not prime can be treated with the same ideas but one has to 
carefully enumerate all the possibilities of contact orders between the branches of 
$E$ and $F$.
\end{rem}

\subsection{A result on the $\delta$-constant stratum through the study of arc spaces} 
\label{resdelta}
$\: $
\medskip

Another classical problem related to the classical study of adjacencies (see section \ref{sec:def_fix}) is the study of the $\delta$-constant stratum of a topological type. 
Using the following result of Teissier (
\cite{T 80}, see also \cite{CL 06}), it consists in studying, among the possible adjacencies, the ones that are realizable by deformations of parametrizations:

\begin{theo}\label{theo:Teissier_delta}
    A family of curves on $(\CC^2,0)$ 
    admits a parametrization in family if and only if it is $\delta$-constant. 
\end{theo}

A family of parametrizations is a convergent wedge realising the corresponding adjacency. Observe that a representative of the wedge may have reducible generic curves at the origin. This means that a wedge can have \emph{returns} (see \cite{tesis}):

\begin{defn}\label{def:returns} 
     The \emph{returns} of a wedge $\alpha:\CC^2\to \CC^2$ are the branches of  
      $(\alpha^{-1}(0), 0)$ different from $\Lambda\times\{0\}$ where $\Lambda$ 
      is the parameter space of the family.
\end{defn}

In the general surface case, the existence of returns can be essential and unavoidable, 
see \cite{tesis}. This is not the case in $\CC^2$: 

\begin{lem}\label{lem:ret} 
   Given a wedge $\alpha$ strictly realizing a Nash-adjacency 
    $\overline{N}_F\contneq\overline{N}_E$, there exists a wedge strictly realizing 
    it and without returns. 
\end{lem}

\begin{proof} 
Write $\alpha(t,s)=(\alpha_1(t,s),\alpha_2(t,s))=\CC^2\to \CC^2$. Picking appropriate 
coordinates of the source $\CC^2$, we can
get parametrizations of the branches of the 
zero set of $\alpha_1:\CC^2\to \CC$ in the form $(q_i(s),s^{k_i})$, for $i$ varying 
in some index set $I$. 
We can always redefine $\alpha(t,s):=\alpha(t,s^k)$ with $k=lcm\{k_i \:  , \:  i \in I\}$, 
which is a wedge with the same properties as the original one. So, we can assume that 
the parametrizations of the zero set of $\alpha_1$ are now of the form $(q_i(s),s)$. 
 
A branch $(q_i(s),s)$ gives a return of the wedge if and only if it is contained in the zero 
set of $\alpha_2$, and if $q_i(s)$ is not identically zero.

Let us consider a number $M \in \NN^*$ such that 
$\alpha_2(q_i(s),s)$ is different from $s\cdot q_i(s)^M$ for all $i$.

Then, the wedge given by $\beta(t,s)=(\alpha_1(t,s),\alpha_2(t,s)+st^M)$ does not have returns. 

Besides, by taking $M$ big enough, we can guarantee that the generic arc of $\beta$ 
is also transversal to the divisor $E$, provided the generic arc of $\alpha$ is transversal to $E$.
\end{proof} 

Now, we can state our main result about $\delta$-constant deformations:

\begin{prop}\label{prop:delta}
Let $E$ and $F$ be satellite prime divisors (that means that give topological types). Let
$V (h_E)$ be a curve associated to E and assume it has only one Puisseux pair (which means that $E$ is toric). If the inequality $\nu_E\leq \nu_F$ holds, then there is a 
$\delta$-constant deformation with generic topological type a curvetta for $E$
     and special curve a curvetta for $F$.

\end{prop}

\begin{proof}
Lemma 3.11 
in \cite{I_max} says that in our case $\nu_E\leq \nu_F$ implies
$\overline{N}_F\cont\overline{N}_E$. Then, by Theorem \ref{theo:wedge_exist} we have a convergent wedge 
realizing the adjacency and by the previous proposition we can suppose that the wedge doesn't have returns and then gives a deformation of parametrizations whose generic type is the one associated to $E$.      
\end{proof}

We have also this easy consequence:
\begin{cor} Let $E$ and $F$ be satellite prime divisors such that $\nu_E\leq \nu_F$ holds. Then $\delta(h_E)\leq \delta(h_F)$ for $h_E$ and $h_F$ functions associated to $E$ and $F$ respectively.
\end{cor}

Proposition \ref{prop:delta} can be seen as a generalization of Alberich and Roe's \cite[Corollary 2.4]{AR}, 
which only deals with classical adjacencies.

\section{Inclusions of maximal divisorial sets in the arc space of $\CC^2$}
\label{sec:arcs}

In this section we prove Theorem \ref{theo:main2} of the introduction and we 
derive some consequences of it. Before entering in the core of the proof, we need 
to recall a few known results about wedges and adjacencies and to introduce 
{\em dicritical} wedges.

\subsection{Nash adjacency for $(E,F)$ is a combinatorial property of the pair $(E,F)$}  \label{sec:combNash}

$\ $
\medskip

Let $(X,O)$ be any normal surface singularity. In this paper we talk only about the case 
$X=\CC^2$,  
but some of the results we are using were proved in the possibly singular setting, 
therefore we state them  in that generality.
\begin{defn}
      An \emph{arc through the origin} $O\in X$ is a morphism 
             \[\mathrm{Spec}(\mathbb{C}[[t]])\to X\]
       sending the special point to the origin $O$. We denote by $\mathcal{X}_\infty$ 
     the space of arcs through the origin of $X$. 
      An arc is said to be \emph{convergent} if the power series that define it are convergent. 
\end{defn}

The space of arcs $\mathcal{X}_\infty$ may be endowed canonically with a structure 
of affine scheme over $X$, as inverse limit of jet spaces at all truncation orders.

We will only consider arcs through the origin, so an \emph{arc} will be always through the origin. 

As when $(X,O)$ is smooth, we say that $S$ is a \emph{model} of $(X,O)$ if: 
      \[\pi:(S,E)\to (X,O)\]
is a proper birational map and $S$ is smooth. The divisor $E$ is the preimage 
of $O$ and is called the \emph{exceptional divisor}. 
Let $E=\cup_{k=1}^rE_k$ be the decomposition of $E$ into irreducible components. 
As in Section \ref{sec:not}, we can talk about the \emph{minimal model} 
in which a divisor $D$ appears. 

Given any model $S$ of $X$, by the valuative criterion of properness, any non-constant arc 
$\gamma$ admits a unique lifting: 
         \[\widetilde{\gamma}:\mathrm{Spec}(\mathbb{C}[[t]])\to S.\]
Of course, if $\gamma$ is convergent, so is $\widetilde{\gamma}$. 

\begin{defn}
    An arc is said \emph{to have a transverse lifting} to $S$ if its lifting $\widetilde{\gamma}$ to 
   $S$ meets only one irreducible component $E_k$ transversely at a smooth point of $E$. 
   An arc is said \emph{to have a transverse lifting by a divisor $F$} if its lifting to the minimal 
    model of $F$ meets $F$ transversally at a smooth point of the exceptional divisor.
\end{defn}

Given a prime divisor $E_k$ and any model $(\widetilde{X},E)\to (X,O)$ where it appears, we can consider the following subsets of the space of arcs:  
\[N_{E_k}=\{\gamma\in \mathcal{X}_\infty: \widetilde{\gamma}(0)\in E_k\},\]
\[\dot{N}_{E_k}=\{\gamma\in N_{E_k}: \widetilde{\gamma}\mathrm{\ is \, transverse\ to}\, E_k \}.\]
The Zariski closure $\overline{N}_{E_k}$ of $N_{E_k}$ in $\mathcal{X}_\infty$  
 is called the \emph{maximal divisorial set associated with $E_k$} 
 and does not depend on the model containing $E_k$ which is considered. 

These sets were  introduced by Nash in \cite{Nash}. We collect their basic properties 
in the following lemma: 

\begin{lem}\label{lem:Nash_sets} (see \cite{Nash} and \cite{Bob} for more details)
We have the following statements:
\begin{enumerate}
\item  The sets $N_E$ and $\overline{{N}}_E$ are irreducible sets of $\calX_\infty$. 
\item  The subset $\dot{N}_E$ is open and dense in $N_E$, therefore 
     $\overline{N}_E$ is also the closure of $\dot{N}_E$.
\item The sets $\dot{N}_E$ are cylindrical (there exists an $m \in \NN$ such 
    that they are the preimages of constructible 
    sets of $m$-jets by the $m$-truncation map).  
\item  The sets $\overline{N}_E$ are of finite codimension.
\item If $(X,O)$ is smooth, then the sets $\overline{N}_E$ are cylindrical.
\item  If $E\neq F$ then $\overline{N}_E\neq \overline{N}_F$. 
\end{enumerate}
\end{lem}
\begin{proof}
     Only (5) is not directly contained in the literature. 
     The sets $\overline{N}_E$ are closures of cylindrical sets,
      and since $(X,O)$ is smooth they are also cylindrical.
\end{proof}

\begin{defn} 
      Given two prime divisors $E$ and  $F$, a \emph{Nash-adjacency} is an inclusion 
      $\overline{N}_F\cont \overline{N}_E$. 
    This defines a partial order $E \leq_N F$ among prime divisors. 
\end{defn}

Note that in previous articles (see \cite{Bob},  \cite{tesis}), what we denote here as Nash-adjacencies were denoted there by \emph{adjacencies}.

Observe that a Nash-adjacency $\overline{N}_F\contneq \overline{N}_E$ implies that the 
topological type of curves associated to $F$ is adjacent (in the sense of section \ref{sec:def_fix}) to the one of 
curves associated to $E$ (in particular, $E\leq_\nu F$).


We will study inclusions of maximal divisorial sets via a characterization using \emph{wedges}, 
which are uni-parametric families of arcs: 
\begin{defn}\label{def:wedge}
      A \emph{wedge} is a morphism: 
             \[\alpha:\mathrm{Spec}(\mathbb{C}[[t,s]])\to X\]
       such that the closed subset $V(t)$ of $\mathrm{Spec}(\mathbb{C}[[t,s]])$  
       is sent to the origin $O\in X$. 

    A wedge is said to be \emph{convergent} if the power series that define it are convergent. 
    In this case we define $\alpha_s:=\alpha|_ {\CC\times \{s\}}$. It is known that we can choose 
    a small enough representative $\Lambda$ of $(\CC,0)$ such that for 
    $s\in\Lambda\setminus\{0\}$, the topological type of the curve parametrized 
   by $\alpha_s$ is constant.  We say that $\alpha_0$ and $\alpha_s$ with 
   $s\in\Lambda\setminus\{0\}$ are respectively \emph{the special arc} and the 
   \emph{generic arcs}  of the wedge. 

   A wedge is said to be \emph{algebraic} if for every $i$, the series $\alpha_i(s,t)$ that 
   define it are algebraic power series,  that is, for any $i$ there exists a polynomial 
    $G_i(t,s,x)$ such that $G_i(t,s,\alpha_i(s,t))=0$.
\end{defn}

More precisely, we are interested in the following wedges (compare with the last 
paragraph of Definition \ref{def:def_fix}): 
 \begin{defn} A \emph{wedge realizing a Nash-adjacency} 
       $\overline{N}_F\contneq\overline{N}_E$ is a wedge $\alpha$  
       such that $\alpha_0$ is in $\dot{N}_F$ and $\alpha_s$ is in ${N}_E$ for $s \neq 0$. 
      Moreover: 
       \begin{enumerate}
              \item we say that $\alpha$ is \emph{dicritical} if and only if $E$ is dicritical for the wedge, 
                   that is, if and only if:  
                         \[ \alpha_s:\mbox{Spec}\ (\CC((s))[[t]]) \to X \] 
                    sends the special point to the 
                   generic point of $E$. In the convergent case, this may be interpreted as the fact 
                   that $\widetilde{\alpha}_s(0)$ moves along the divisor $E$, 
                   when the parameter $s$ varies.
              \item we say  that $\alpha$ \emph{strictly realises the Nash-adjacency} 
                 if and only if $\alpha_s$ lifts transversally to $E$ for all $s\neq 0$ small enough.  
     \end{enumerate}
\end{defn}

Observe that a wedge realising a Nash-adjacency neither has to be dicritical nor has to 
strictly realize an adjacency. There is also no implication between being dicritical and 
strictly realizing an adjacency.

The following theorem was fundamental to prove the Nash Conjecture in \cite{Nash_surfaces} and will be used in the proof of Theorem \ref{theo:main2}. It was originally stated for 
any 2 divisors appearing in the minimal resolution of a singular surface germ but it works in general for any two divisors appearing in a model 
(which is always the case of a pair of divisors) It follows from the results in \cite{Reg3} and \cite{Bob} and can be found in \cite[Thm. A, Prop. 17]{Bob}: 

\begin{theo}\label{theo:wedge_exist}(see \cite{Bob}) 
       Let $E$ and $F$ be two prime divisors over $O\in X$. Then, the following assertions 
       are equivalent: 
\begin{enumerate}
    \item one has the Nash adjacency $\overline{N}_{F}\contneq\overline{N}_E$; 
    \item there exists a wedge $\alpha$ realising the Nash-adjacency 
          $\overline{N}_{F}\contneq\overline{N}_E$; 
    \item for any convergent $\gamma\in \dot{N}_F$, there exists an algebraic wedge 
          $\alpha$ with $\alpha_0=\gamma$, strictly realizing the 
        Nash-adjacency $\overline{N}_{F}\contneq\overline{N}_E$. 
\end{enumerate}
\end{theo}

\begin{rem}
 In the proof of $(1)\Rightarrow (2)$, a version of the Curve Selection Lemma from \cite{Reg3} was used. For the case of $\CC^2$ we are interested in, we can substitute it by a simpler proof since the arc space of $\CC^2$ is an infinite dimensional affine space.
 \end{rem}

The following new lemma will be used: 
\begin{lem}\label{lem:dic_wedge}
    Let $\alpha$ be a wedge realising a Nash adjacency $\overline{N}_F\cont \overline{N}_E$. 
    Then there exists $E'\geq_d E$ such that $\alpha$ realizes the Nash-adjacency 
     $\overline{N}_{F}\cont\overline{N}_{E'}$ and is dicritical for $E'$. 

    Besides, if $\alpha$ strictly realizes the Nash-adjacency 
    $\overline{N}_F\cont \overline{N}_E$, then it also strictly realizes the Nash-adjacency 
    $\overline{N}_F\cont \overline{N}_{E'}$.
\end{lem}

\begin{proof}
Consider the minimal model of $E$. If the wedge is not dicritical we can assume that there 
is a point $p\in E$ where all the generic arcs $(\alpha_s)_{s \neq 0}$ meet. 
We blow-up $p$ and keep blowing up points $p_1\in E_p$, $p_2$,..,$p_n$ if the liftings 
of the generic arcs still meet at $p_i\in E_{p_{i-1}}$. After a finite number of blow ups, 
the intersection multiplicity $I_{p_i}(\widetilde{\alpha}_{s_1}, \widetilde{\alpha}_{s_2})$ 
goes down and then after another finite number of blow ups we make the wedge dicritical.  
\end{proof}

We prove now Theorem \ref{theo:main2} of the introduction, that is, that the existence 
of an inclusion $\overline{N}_F\contneq\overline{N}_E$  only depends on the combinatorics 
of the pair $(E,F)$. It is based on a similar result in \cite{Bob} (Theorem D) and we 
quote \cite{Bob} for technical aspects that are fully explained there.
Using Theorem \ref{theo:wedge_exist}, we see that it  is equivalent to the following: 
 
\begin{theo}\label{theo:Nash_top} 
      Let $(X,O)$ be smooth. 
      Assume that there exists a wedge $\alpha$ realizing the Nash-adjacency 
      $\overline{N}_F\contneq\overline{N}_E$. If $(E',F')\equiv(E,F)$, then there exists a 
      wedge $\beta$ realizing the Nash-adjacency $\overline{N}_{F'}\contneq\overline{N}_{E'}$.
\end{theo}

\begin{proof}
    Identify $(X,O)$ with $(\CC^2,0)$ by a choice of local coordinates. 

The theorem is trivial if $E\leq_dF$, because in this case $F'$ also dominates 
$E'$, which implies that $\overline{N}_{F'}\contneq\overline{N}_{E'}$.  Therefore, 
we can assume that this is not the case.

Using point (3) of Theorem \ref{theo:wedge_exist}, we can assume that 
  $\alpha:(\CC^2,0)\to (\CC^2,0)$ is an algebraic wedge that strictly realizes 
  the adjacency for $(E,F)$.

We can reduce it to the case of $\alpha$ a dicritical wedge as follows. If $\alpha$ is not 
dicritical for $(E,F)$, following Lemma \ref{lem:dic_wedge}, we can take $E_2>_d E$ 
such that it is dicritical for $(E_2,F)$. Now, we take $E_2'>E'$ such that 
$(E_2',F')\equiv (E_2,F)$. If we can find $\beta$ realizing the adjacency for $(E_2',F')$, 
then it results from the domination $E_2'>_d E'$  that $\beta$ also realizes the adjacency 
for $(E',F')$.  

So, let us assume that $\alpha$ is dicritical. 
 
Let: 
     $$\pi_{E,F}: S_{E,F}\to (\CC^2,0)$$ 
   be the minimal model of $E+F$. Then, since $E\nleq_dF$, we have that the liftings  
   to $S_{E,F}$ of $\alpha_0$ and $\alpha_s$ for small enough $s\neq 0$  
   hit $F$ and $E$ respectively.

By Proposition 22 in \cite{Bob} there is a smooth algebraic surface $\Sigma$, 
a proper morphism: 
    $$\Phi: \Sigma \to \CC^2$$ 
and an embedding: 
       $$i:(\CC^2,0)\to \Sigma$$ 
  such that the following diagram commutes: 
\begin{equation}\label{dia:0}\xymatrix{
\Sigma\ar^{\Phi}[rr]&&(\CC^2,0)\\
(\CC^2,0)\ar^{i}[u]\ar^{\alpha}[urr]&&
}
\end{equation}
In other words, the wedge $\alpha$ is the germ of the mapping $\Phi$ around a point 
$p:=i(0)\in \Sigma$ and we can think about the coordinates $(s,t)$ as local coordinates around 
$p$. Note that  the map $\Phi$ is not necessarily a model above $(\CC^2,0)$, as it 
is not always bimeromorphic (it may have generic degree $\geq 2$).   
We denote by $L_s$ and $L_t$ the images by $i$ of the $s$-axis (defined by $t=0$) 
and the $t$-axis of $(\CC^2)_{(s,t)}$ respectively. We denote by $\Omega$ 
a small ball around $p$. The arc $\alpha_0$ is identified with $\Phi|_{L_t}$. 
Its lifting to $S_{E,F}$ is denoted by $\tilde{\alpha}_0$.
  
Let $U$ be a Milnor ball for the image of $\alpha_0$ around the origin of $\CC^2$. Consider 
$\widetilde{U}:=\pi_{E,F}^{-1}(U)$, which is a small tubular neighborhood at the exceptional 
divisor of $\pi_{E,F}$.
 
Define $V:=\Phi^{-1}(U)$. We assume that $\Omega$ was chosen such that $\Omega\:  \cont \:  V$. 
Consider the fiber product $\overline{V} := V\times_U\widetilde{U}$.
Denote by $\sigma_1$, $\sigma_2$ its projections to the each of the factors. 
Let $\rho:\widetilde{V}\to \overline{V}$ be the minimal resolution of $\bar{V}$. 
Then $\widetilde{\Phi}:=\sigma_2\:   \comp \:   \rho$ is a proper map:  

\begin{equation}\label{dia:1}\xymatrix{
\widetilde{V}\ar^{\kappa}[ddrr]\ar^{\widetilde{\Phi}}[drrrr] \ar^{\rho}[rrd]&&&&\\
&&\overline{V}\ar^{\sigma_1}[d]\ar^{\sigma_2}[rr]&&\widetilde{U} \:  
    \cont \:  S_{E,F}\ar^{\pi_{E,F}}[d]\\
&&V\ar^{\Phi}[rr]&&U   \:  \cont   \: \CC^2\\
&&(\CC^2,0)\ar^{i}[u]\ar_{\alpha}[urr]&&
}
\end{equation}
We denote by $\widetilde{L}_s$ and $\widetilde{L}_t$ the strict transform by $\kappa$ of $L_s$ and $L_t$ in  $\widetilde{V}$. 

Denote $A:=\widetilde{\Phi}^{-1}(\pi_{E,F}^{-1}(0))$ and consider its decomposition 
into irreducible components: 
    $$A=\cup_iA_i.$$
Consider the total transform of $L_s+L_t$ in $\widetilde{V}$ by $\kappa$, and its decomposition into irreducible components. Then: 
\begin{enumerate}
\item there is a unique non-compact irreducible component that contains $\widetilde{L}_t$;
\item we have that $\widetilde{\alpha}_0 \:  \comp  \:  \kappa|_{\widetilde{L}_t}= 
     \widetilde{\Phi}|_{\widetilde{L}_t}$; 
\item the rest of components are contained in $A$ (and are compact); 
\item $\widetilde{L}_t$ is transverse to a compact component; 
\item there is only one compact component that contains $\widetilde{L}_s$; we denote it by $A_s$; 
\item since $\alpha$  is a dicritical wedge, $A_s$ is dicritical for $\widetilde{\Phi}$ and $\widetilde{\Phi}(A_s)=E$. 
\end{enumerate}

Take the Stein factorization of $\widetilde{\Phi}$: 
\begin{equation}\label{dia:2}
    \xymatrix{
\widetilde{V}\ar^{\theta}[rrrr]\ar^{\widetilde{\Phi}}[rrrrd] &&&&W\ar^{\Psi}[d]\\
&&&&\widetilde{U}}
    \end{equation}

where $\theta$ is proper and birational, $W$ normal and $\Psi:W\to \widetilde{U}$ finite. Then $\Psi$ is a branched cover. 

Let $R \:  \cont \:  W$ be the \emph{ramification locus} of $\Psi$. 
 
We have the following: 
\begin{enumerate}
\item a divisor $A_i$ is dicritical for $\widetilde{\Phi}$ if and only if $A_i$ 
      is not collapsed  by $\theta$;
\item the divisor $A_s$ (that contains $\widetilde{L}_s\cont \widetilde{V}$) is 
       dicritical for $\widetilde{\Phi}$, thus, it is not collapsed by $\theta$ and $\theta(A_s)$ 
       is a divisor in $W$; 
\item since the lifting $\widetilde{\alpha}_s$ of the arc $\alpha_s$ is transverse to $E$, 
      the curve $\theta(A_s)$ is not contained in the ramification locus, that is 
      $\theta(A_s)\nsubseteq R$. 
\end{enumerate}

Define the \emph{branching locus} $\Delta:=\Psi(R)$. 
We decompose $\Delta=\Delta^{exc}\cup \Delta'$, where $\Delta^{exc}$ is the part collapsed by $\pi_{E,F}$, that is \emph{the exceptional part}. 

Let $\pi_{E',F'}:S_{E', F'}\to \CC^2$ be the minimal model of $E'+F'$. 
Then, $\pi_{E',F'}^{-1}(0)$ has the same dual graph as $\pi_{E,F}^{-1}(0)$. 
Let $\widetilde{U}'$ be a tubular neighbourhood of $\pi_{E',F'}^{-1}(0)$. 

As in \cite[Section 6, Proof of Prop. 25 Part I or Proof of Prop. 32]{Bob}, 
we construct a diffeomorphism: 
$$\tau: \widetilde{U}\to \widetilde{U}'   \:   \cont \:  S_{E',F'}$$ such that: 
\begin{enumerate}[(i)]
   \item $\tau(E)=E'$, $\tau(F)=F'$ , $\tau(\pi_{E,F}^{-1}(0))=\pi_{E',F'}^{-1}(0)$
   \item $\tau$ is a biholomorphism in neighbourhoods of the singularities of 
     $\pi_{E,F}^{-1}(0)$, the points where $\Delta'$ meet $\pi_{E,F}^{-1}(0)$ and the point 
     where $\widetilde{\alpha}_0$ meets $F$. 
\end{enumerate} 
Note that $(ii)$ implies that $\beta_0:=\pi_{E',F'}\comp\:  \tau\:  \comp \: \widetilde{\alpha}_0$ is a holomorphic arc whose lifting to $\tilde{U}'$  is transversal to $F'$ and equals 
$\tau \:   \comp   \:   \widetilde{\alpha}_0$. We denote the lifting by $\tilde{\beta}_0$ 
(i.e. $\tilde{\beta}_0=\tau\:  \comp \:  \tilde{\alpha}_0$). 

Define the composition: 
     $$\lambda:=\tau \:  \comp \:  \Psi.$$ 
The morphism $\lambda$ is 
unramified over $\widetilde{U}'\setminus \tau(\Delta)$ where $\tau(\Delta)$ is a 
complex analytic curve in $\widetilde{U}'$, since $\tau$ satisfies (ii). 


Then, using a theorem of Grauert and Remmert  (see \cite{GRinit}, \cite{GR} or \cite{BN}), 
as in the proof of \cite[Lemma 28]{Bob}, we deduce the existence of a unique 
normal holomorphic structure in $W$ which makes $\lambda$ a complex analytic morphism. 
We denote by $W'$ the variety $W$ with this new complex structure, by $\overline{\tau}$ 
the identity diffeomorphism (as topological spaces) from $W$ to $W'$ and by $\Psi'$ 
the complex analytic morphism from $W'$ to $\widetilde{U'}$ given by $\lambda$ 
(or more precisely  $\lambda  \:  \comp  \:  \overline{\tau}^{-1}$): 

\begin{equation}\label{dia:3}\xymatrix{
W\ar^{\lambda}[rrd]\ar^{\Psi}[d]\ar^{\bar{\tau}}[rr]&&W'\ar^{\Psi'}[d]\\
\widetilde{U}\ar^{\tau}[rr]&&\widetilde{U}'
}\end{equation}

In the left handside of diagram (\ref{dia:paralelo}),  one can find the diagrams 
(\ref{dia:1})-(\ref{dia:2}) that we have associated with the wedge $\alpha$. 
Diagram (\ref{dia:3}) appears in the center. In the rest of the proof we will complete the right handside of diagram~(\ref{dia:paralelo}).

Denote $U':=\pi_{E',F'}(\widetilde{U}')$. 

First observe that since $\widetilde{V}$ is smooth, the mapping $\theta:\widetilde{V}\to W$ 
can be factorised as: 
\begin{equation}\label{dia:4}\xymatrix{
&&&&\widetilde{W}\ar^{\mu}[d]\\
\widetilde{V}\ar^{\nu}[urrrr]\ar^{\theta}[rrrr]&&&&W
}\end{equation}
where $\mu$ is the minimal resolution of $W$ and $\nu$ is a sequence of blow ups of points. 

Let $\mu':\widetilde{W}'\to W'$ be the minimal resolution of $W'$. 
Since $\tau $ is a homeomorphism and the combinatorics of the minimal resolution of a surface singularity only depends on the topology of the link, the dual graphs of the divisors $\mu^{-1}(\pi_{E,F}^{-1}(0))$ and $\mu'^{-1}(\pi_{E',F'}^{-1}(0))$ are the same. 

We define $\widetilde{V}'$ as the space obtained from $\widetilde{W}'$ by a sequence of blow ups $\nu':\widetilde{V}'\to \widetilde{W}'$ such that the dual graph of $B':=\nu'^{-1}(\mu'^{-1}(\pi_{E',F'}^{-1}(0))$ is isomorphic to the dual graph of $B:=\nu^{-1}(\mu^{-1}(\pi_{E,F}^{-1}(0))$. 

Define $\theta':=\nu'  \:  \comp \:  \mu'$ and $\widetilde{\Phi}':=\Psi'   \:  \comp \:   \theta'$.

\begin{equation}\label{dia:paralelo}\xymatrix{
\widetilde{V}\ar@/^1.5pc/[rrrrrrrrrr]^{\overline{\overline{\tau}}} \ar^{\kappa}[ddrr] \ar^{\widetilde{\Phi}}[drrrr] \ar^{\rho}[rrd]\ar^{\theta}[rrrr]&&&& W\ar^{\overline{\tau}}[rr]\ar^{\Psi}[d] \ar^{\lambda}[drr]&& W'\ar_{\Psi'}[d]&&&&\widetilde{V}'\ar_{\theta'}[llll]\ar_{\widetilde{\Phi}'}[lllld]\ar^{\xi_1}[dlll]\ar^{\kappa'}[ddll] \\
&&\overline{V}\ar^{}[d]\ar[rr]&&\widetilde{U} \:  \cont \:  S_{E,F}\ar^{\pi_{E,F}}[d]\ar^{\tau}[rr]&&\widetilde{U}'\ar_{\pi_{E',F'}}[d]& T\ar^{\xi_2}[dl]&&&\\
&&V\ar^{\Phi}[rr]&&U  \:  \cont   \:  \CC^2&&U'&&V'\ar^{\upsilon}[lu]\ar_{\Phi'}[ll]&&
}
\end{equation}

Notice that the pairs of complex spaces: 
$$(W,\Psi^{-1}(\pi_{E,F}^{-1}(0)\cup\widetilde{\alpha}_0))$$ 
$$(W',\Psi'^{-1}(\pi_{E',F'}^{-1}(0)\cup\widetilde{\beta}_0))$$
are diffeomorphic because $\tau$ sends $\widetilde{\alpha}_0$ to $\widetilde{\beta}_0$. 
Hence the pairs: 
\begin{eqnarray}
\label{eq:tildeV}
(\widetilde{V}, \widetilde{\Phi}^{-1}(\pi_{E,F}^{-1}(0)\cup \widetilde{\alpha}_0))\\
(\widetilde{V}', \widetilde{\Phi}'^{-1}(\pi_{E',F'}^{-1}(0)\cup \widetilde{\beta}_0))\end{eqnarray}
are diffeomorphic. To construct a diffeomorphism $\overline{\overline{\tau}}$ from  $\widetilde{V}$ to $\widetilde{V}'$ 
that makes the diagram~(\ref{dia:paralelo}) commutative, we first lift $\overline{\tau}$ 
to a diffeomorphism  of the minimal resolutions of $W$ and $W'$ obtaining 
$\overline{\tau}':\widetilde{W}\to \widetilde{W}'$. Now, recall that $\widetilde{V}$ and $\widetilde{V}'$ are obtained by a sequence of blow ups with the same combinatorics 
from  $\widetilde{W}$ and $\widetilde{W}'$. Then, we lift $\overline{\tau}'$ step by step following this sequence of the blow ups, in both origin and target, starting 
from $\widetilde{W}$ and $\widetilde{W}'$ respectively. It is clear we can lift it for a single blow up, so finally we can find a diffeomorphism 
$\overline{\overline{\tau}}:\widetilde{V}\to \widetilde{V}'$ such that the diagram commutes. 

In particular, a divisor is collapsed by $\widetilde{\Phi}$ if and only if its image by $\overline{\overline{\tau}}$ is collapsed by $\widetilde{\Phi}'$. Also a divisor $D_1$ in $\widetilde{V}$ is dicritical onto some divisor $D_2$ in $\pi_{E,F}^{-1}(0)$ if and only if  $\overline{\overline{\tau}}(D_1)$ is dicritical onto $\tau(D_2)$ in $\pi_{E',F'}^{-1}(0)$.

Let $C$ be the divisor that is collapsed by $\kappa:\widetilde{V}\to V$ and let $C'$ be its image by $\overline{\overline{\tau}}$. 
Then, since $\overline{\overline{\tau}}$ is a diffeomorphism, we can collapse $C'\cont \widetilde{V}'$ and obtain a smooth surface $V'$. 
We call this map $\kappa':\widetilde{V}'\to V'$.  

Observe that a divisor in $\widetilde{V}'$ is collapsed by $\kappa'$ if and only if it is the image by $\overline{\overline{\tau}}$ of a divisor in $\widetilde{V}$ that is collapsed by $\kappa$ respectively.

Let ${L}_t'$ and ${L}_s'$ be the images of $\widetilde{L}_s$ and $\widetilde{L}_t$ by 
$\kappa'  \:  \comp \:  \overline{\overline{\tau}}$. The union meets
in a normal crossing in $V'$ because they are not contracted by $\kappa'$, $\overline{\overline{\tau}}$ is a diffeomorphism and the smoothness and transversality depend only on 
the combinatorics of the divisor $C'$ that we collapse (that is the same as the one of $C$).

Applying Stein factorization to the composition 
$\pi_{E',F'}\:  \comp \:  \widetilde{\phi}':\widetilde{V}'\to U'$ we obtain a normal space $T$, 
a proper and birational mapping $\xi_1:\widetilde{V}'\to T$ and a finite mapping $\xi_2:T\to U'$.

The mapping $\widetilde{V}'\to T$ is the result of collapsing the divisor 
   $ \widetilde{\Phi}'^{-1}(\pi_{E',F'}^{-1}(0))$. 
Since $\kappa':\widetilde{V}'\to V'$ collapses $C'$ which is a union of components of 
    $\widetilde{\Phi}'^{-1}(\pi_{E',F'}^{-1}(0))$, then there is a morphism $\upsilon:V'\to T$. 

Define $\Phi'=\xi_2\:  \comp \:  \upsilon$ and let  $\beta$ be the germ of $\Phi'$ at the point 
   $p':=L_s'\cap L_t'$. 
Choose coordinates $(t,s)$ around $p'$ so that the $t$-axis is $L_t'$ and the $s$-axis is $L_{s}'$.

We see that $\beta$ is a dicritical wedge realizing the Nash-adjacency for $(E', F')$. 
We  think about $\beta$ as being defined in 
$\Omega':=\kappa'(\overline{\overline{\tau}}(\kappa^{-1}(\Omega)))$.

Indeed, by construction we have that: 
\begin{itemize}
\item The curve $L_{t}'$ lifts to $\widetilde{U}'$ to the lifting of $\beta_0$, that is, to 
     $\tau\comp\widetilde{\alpha}_0$ which is transversal to $F'$ at a smooth point of the 
     exceptional divisor.   
\item The curve $L_s'$ lifts to $\widetilde{V}'$ by $\kappa'$ to $\overline{\overline{\tau}}(A_s)$,  
     which is transformed by $\widetilde{\Phi}'$ to $E'$ in a dicritical way. Thus the special point 
     of the lifting of the generic arc moves along $E'$ when $s$ moves. 
\item The lifting $\widetilde{\beta}_s$ meets $E'$ transversely because $\widetilde{\Phi}'$ 
     is unramified at $\widetilde{L}_s'$ (this is because $\widetilde{\Phi}$ is unramified at 
     $L_s\cont A_s$ and $\widetilde{\Phi}'$ is obtained from $\widetilde{\Phi}$ by composition 
     with a diffeomorphism).
\end{itemize}
\end{proof}

\subsection{Consequences of Theorem \ref{theo:Nash_top} and some conjectures}\label{sec:cyl}

$\ $
\medskip 

In this section, we use Theorem \ref{theo:Nash_top} to look for necessary or sufficient conditions 
for the existence of Nash-adjacencies. 
\medskip

Let $E$, $E'$ be prime divisors.
We write $E\equiv_{\geq i} E'$ to say that $E$ and $E'$ have the same combinatorics 
and that their contact 
order is greater or equal than $i$.  

We denote by $\overline{Cont^q(E)}$ the closure of the set of arcs  whose lifting 
to the minimal model 
of $E$ meet $E$ with intersection multiplicity $q$. 
Following the terminology of \cite{ELM} or \cite{I_max}, it is the 
\emph{maximal divisorial set} associated with the valuation $q\cdot\nu_E$. 

\begin{lem}\label{lem:cyl1} 
      Given a prime divisor $E$, for any $i$, the intersection: 
           $$\bigcap_{E'\equiv_{\geq i} E} \overline{N}_{E'}$$ 
      is a closed cylindrical set. Moreover, there exists a finite set of divisors 
      $\{E_j\}_{j\in J}$ satisfying  $E_j\equiv_{\geq i}E$ for $j\in J$ such that:  
            $$\bigcap_{E'\equiv_{\geq i} E} \overline{N}_{E'}=\bigcap_{j\in J} \overline{N}_{E_j}.$$ 
\end{lem}

\begin{proof}
The sets $\overline{N}_{E'}$ are closed and cylindrical for the same order of truncation. 
Then, the intersection occurs at the level of truncation, that is,  in a finite dimensional set, 
where the intersection of an infinite number of algebraic sets is the intersection of only a finite number of them. 
\end{proof}

\begin{cor}\label{cor:Nash_max_cont}
Let $(E,F)$ be a pair of prime divisors such that we have the Nash adjacency 
$N_F\subseteq N_E$.  If $(E',F')$ is 
a pair of divisors with $E'$ and $F'$ combinatorially equivalent to $E$ and $F$ respectively 
and $Cont(E',F')\geq Cont (E,F)$,  then we have the Nash adjacency $N_{F'}\subseteq N_{E'}$.    
\end{cor}

	\begin{proof}
	    Let $i+1$ be the contact order between $E$ and $F$. Let $p_i$ be the 
	last point in common on the process of finding the minimal model for $E$ and $F$. 
	The corollary does not say anything unless the next points in both resolutions are free,     
	because only in this case there exist other divisors $E'$ and $F'$ 
	with the same combinatorics as $E$ and $F$ 
	respectively, but with strictly higher contact than $E$ and $F$.
	
	By Lemma \ref{lem:cyl1}, there is a finite set $J$ such that:  
	$$\bigcap_{j\in J} \overline{N}_{E_j}= \bigcap_{E''\equiv_{\geq i} E} \overline{N}_{E''}.$$
	By Theorem~\ref{theo:Nash_top}, if a divisor $F''$ is such that $(E_j,F'')$ is combinatorially 
	equivalent to $(E,F)$ for any $j\in J$, then we have the inclusion: 
	$$\overline{N}_{F''}\ \ \subset\ \ \bigcap_{j\in J} \overline{N}_{E_j}.$$ We deduce the inclusion:
	
	\begin{equation}\label{eq:top1}
	   \bigcup_{(E_j,F'')\equiv(E,F)\ for\ all\ j\in J} \overline{N}_{F''}\ \ \subset\ \ 
	   \bigcap_{E''\equiv_{\geq i} E} \overline{N}_{E''}.
	\end{equation}
Now, take $F''$ with $F''\equiv_{\geq i} F$ and $(E_j,F'')\equiv (E,F)$ for all $j\in J$. 
Choose $E''$ with $E''\equiv_{\geq i} E$ and such that $(E'',F'')\equiv (E',F')$. By the previous inclusion,  
	we have that $F''$ is upper Nash-adjacent to $E''$. Therefore, by Theorem~\ref{theo:Nash_top}, $F'$ 
	is upper Nash-adjacent to $E'$. 
\end{proof}

\begin{cor}\label{cor:cyl} 
     If $E$ and $F$ have contact order $i$, and $\overline{N}_F\contneq\overline{N}_E$, 
     then we have the following inclusions:   
     $$\overline{N}_F\contneq\bigcup_{F'\equiv_{\geq i} F} \overline{N}_{F'}\cont
     \bigcap_{E'\equiv_{\geq i} E} \overline{N}_{E'}\contneq\overline{N}_E$$
\end{cor}
\begin{proof}
 We have to show that $\overline{N}_{F'}\contneq\overline{N}_{E'}$ for any 
    $F'\equiv_{\geq i} F$ and $E'\equiv_{\geq i} E$. 
 Consider such an $F'$. Choose $E''$ such that $E''\equiv_{\geq i} E$ and that $(E'',F')$ 
 is combinatorially equivalent to $(E,F)$. 
 By Theorem ~\ref{theo:Nash_top}, we have $\overline{N}_{F'}\contneq\overline{N}_{E''}$. 
 Since the contact order of $E'$ and $F'$ is at least $i$ when 
 $E'\equiv_{\geq i} E$, by the previous corollary we have 
   $\overline{N}_{F'}\contneq\overline{N}_{E'}$ 
 for any $E'\equiv_{\geq i} E$.
\end{proof}

\medskip
We state now two conjectures concerning maximal divisorial sets.
 
\begin{conj}\label{conj:cyl1} Let $E$ be a prime divisor. 
    Consider the infinitely near points $\{p_i\}_{i\in I}$ 
which are blown-up in order to obtain the minimal model of $E$. 
Let $p_{i_0}$ be any free point among them. 
We conjecture that the set $\bigcap_{E'\equiv_{i_0} E}\overline{N}_{E'}$ coincides 
with  $\overline{Cont^q(E_{p_{i_0}})}$ for certain $q$. In case the conjecture is true, 
it would be interesting to compute $q$. 
\end{conj}

\begin{rem}
   For the case $i_0=1$, the previous conjecture becomes:   
          $$\bigcap_{E'\equiv_{\geq 1} E}\overline{N}_{E'}= \overline{Cont^{M}(E_{0})}$$ 
for a certain $M$, where $E_0$ is the divisor that appears after 
blowing up the origin of $\CC^2$. By the combinatorial nature of Nash adjacency, 
$M$ must depend only on the Puiseux exponents of a curve $V(h_E)$ associated with $E$.
In the case of one Puiseux pair, Ishii's Lemma~3.11~\cite{I_max} implies 
the equality $M=\beta_1$. This comes from the fact that the divisorial valuation
associated to the last divisor of the minimal embedded resolution of a branch with  $1$ Puiseux 
pair is a toric valuation relative to conveniently chosen coordinates. 
But this equality is false in general. Until now, we only have been 
able to prove that $M$ is bounded from above by $\beta_k-1$.
\end{rem}

\begin{conj}\label{conj:cyl3} 
Assume $Cont(E, F)=i$  and that $p^E_{i+1}$ and $p^F_{i+1}$ are both free points. Consider:   
$$k_F:=max \left\{m:\bigcup_{F'\equiv_i F} \overline{N}_{F'}\contneq
       \overline{Cont^{m}E_{p_{i}}} \right\},$$
$$k_E:=min \left\{m: \overline{Cont^{m}E_{p_{i}}}\contneq
       \bigcap_{E'\equiv_{\geq i} E} \overline{N}_{E'}\right\}.$$ 
We conjecture that  $\overline{N}_F\contneq\overline{N}_E$ if and only if $k_F\geq k_E$. 
It would also be interesting to compute $k_E $ and $k_F$ 
in terms of combinatorial invariants of $E$ and $F$. 
\end{conj}

The previous conjectures are motivated by our feeling 
that solving the generalized Nash problem as stated in the 
introduction seems very close to describing the inclusions among the maximal divisorial 
sets $\overline{Cont^q(E)}$.

\begin{defn} \label{logdis}
    Given a prime divisor $E$,  its \emph{log-discrepancy} $a_E(\CC^2)$ is the 
    positive integer $$1 + \nu_E(dx \wedge dy).$$ Here $\nu_E(dx \wedge dy)$ 
    denotes the order of vanishing along $E$ of the pull-back of the holomorphic 
    $2$-form $dx \wedge dy$ to a model in which $E$ appears.
\end{defn}
     
   One has the following well-known rules of computation of the log-discrepancy, 
   whose proofs are direct: 

\begin{itemize}
		\item If $E_0$ is the divisor that appears after the blow up of the origin of $\CC^2$, 
		then $a_{E_0}(\CC^2)=2$. 
   \item If $E$ is \emph{satellite} and appears after blowing up $p\in E_1\cap E_2$, then 
         $a_E(\CC^2)=a_{E_1}(\CC^2)+a_{E_2}(\CC^2)$.

   \item If $E$ is \emph{free} right after $E_1$ then $a_E(\CC^2)=a_E(\CC^2)+1$.
\end{itemize}
   
As a special case of the result \cite[Theorem 2.1]{ELM} of Ein, Lazarsfeld and Musta\c{t}\u{a} 
(see also \cite[Theorem 6.2]{EM}), we have:

\begin{theo} \label{codimlog}
     The codimension $codim(\overline{N}_E)$ of $\overline{N}_E$ in the arc space 
     of $\CC^2$ is equal 
     to the log-discrepancy $a_E(\CC^2)$. 
\end{theo}

Observe that our subsets $\overline{N}_E$ are always contained in the set of arcs centered at 
the origin, which is a set of codimension 2 in the whole arc space of $\CC^2$. 
We will use now the previous theorem in combination with Theorem \ref{theo:Nash_top} in order 
to improve the inequalities between log-discrepancies (see Corollary \ref{cor:codim} below).

As an immediate consequence of Corollary \ref{cor:cyl}, we get:

\begin{cor}\label{cor:codim0}
If $Cont(E, F)=i$, and $\overline{N}_F\contneq\overline{N}_E$, then: 
     $$codim \left(\bigcup_{F'\equiv_{{\geq i}} F} \overline{N}_{F'}\right)\geq 
       codim \left(\bigcap_{E'\equiv_{{\geq i}} E} \overline{N}_{E'}\right).$$
			
			\end{cor}

\begin{lem}\label{lem:codim}  Let $D$ be a prime divisor that appears after 
			blowing up $n+1$ points $\{x_i\}_{i=0}^n$. Fix $i\in \{0, ..., n-1\}$. 
			Let $k$ be the number of free points 
			$x_j$ with $i\leq j\leq n$. Then: 
			\begin{enumerate}
 \item  if $k\geq 1$, then we have:  
			$$codim \left( \bigcap_{D'\equiv_{{\geq i}} D} \overline{N}_{D'} \right) \geq codim 
			      (\overline{N}_D)  +1.$$			
 \item we have:  $$codim \left(\bigcup_{D'\equiv_{{\geq i}} D} \overline{N}_{D'}\right)=
               codim (\overline{N}_D) -k,$$		
\end{enumerate}
\end{lem}

		\begin{proof}
			Part (1) is obvious since $k\geq 1$ implies that there are at least two different 
			subsets $\overline{N}_{D'}$ intersecting. 
			
			For part (2), notice that 
			$\bigcup_{D'\equiv_{{\geq i}} D}(\overline{N}_{D'})$ 
			fibers over the variety parametrising the 
			possible positions of the 
			free points $x_j$ with $i\leq j\leq n$. This variety has dimension $k$ and 
			its fibre over a point is $\overline{N}_{D'}$ for a certain $D'$ satisfying 
			$D'\equiv_{{\geq i}} D$.
		\end{proof}
			
\begin{cor}\label{cor:codim} 
Assume that $\overline{N}_F\subsetneq \overline{N}_E$. Let $i=Cont(E,F)$. 
Let $k$ be the number of free points after $x_{i}$ (and including it) that we blow up to 
obtain the minimal model of $F$. Let $h$ be equal to $1$ if there is at least one free point 
after $x_i$ 
in the blowing up process leading to the minimal model of $E$ and  be equal to to $0$ otherwise. 
Then the following inequality holds:    
    \begin{equation}\label{eq:mej_disc}
         a_E(\CC^2)<a_F(\CC^2)-k-h.
    \end{equation}
\end{cor}
\begin{proof}
 This is a direct application of the previous lemma.
\end{proof}


\subsection{An Example}\label{sec:ex}  $\: $ 
\medskip

Pl\'enat showed in \cite[2.2]{Cam} that a Nash adjacency 
$\overline{N}_F\subseteq\overline{N}_E$ implies the valuative inequality $\nu_E\leq \nu_F$.

In \cite{I_max}, Ishii gave an example of a pair $(E, F)$ where the valuative inequality $\nu_E\leq \nu_F$  
holds, but $\overline{N}_F\ncontneq \overline{N}_E$. 
The example was for $E$ and $F$ as in Figure \ref{fig:ishii}. 
 Since  $a_E(\CC^2)=9$ and $a_F(\CC^2)=8$, we can use instead 
 Corollary \ref{cor:codim} to see directly that 
 the Nash-adjacency $\overline{N}_F\subseteq\overline{N}_E$ does not hold. 

Observe that in Ishii's example, the divisor $E$ is not a satellite divisor (that is, it is 
not the final divisor of the minimal embedded resolution of a plane branch). 
One could still ask oneself 
whether the valuative inequality implies the Nash-adjacency for satellite divisors. 
In the following example we will see that this is not true.

\begin{figure}%
\includegraphics[width=30mm]{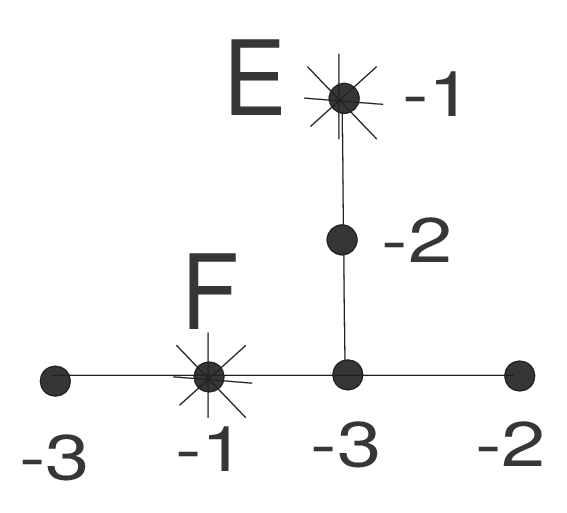}
\caption{Dual graph of the exceptional set of the minimal model of $E+F$ in Ishii's example.} 
\label{fig:ishii}%
\end{figure}

Namely, we exhibit two satellite divisors $(E,F)$ with contact order 1 
for which we have $\nu_E\leq \nu_F$, $a_E(\CC^2)<a_F(\CC^2)-k-h$ 
as in Corollary \ref{cor:codim},  but $E\nleq_N F$. That is, 
$\overline{N}_F\ncontneq \overline{N}_E$. 

The example is the following. Take the curve $V(h_E)$ parametrized by $(t^9,t^6+t^8)$. 
The curve has multiplicity sequence $(6,3,3,2,1,1)$ and 
characteristic exponents  $(6,9,11)$. Consider $E$ with the combinatorial type of the 
final divisor of the minimal embedded resolution of $V(h_E)$. Take $F$ a divisor whose 
associated curves have multiplicity sequence $(10,1,1,...,1)$. We consider divisors $E$ 
and $F$ of these combinatorial types and contact order $1$ as in Figure \ref{fig:ex}.

\begin{figure}%
\includegraphics[width=90mm]{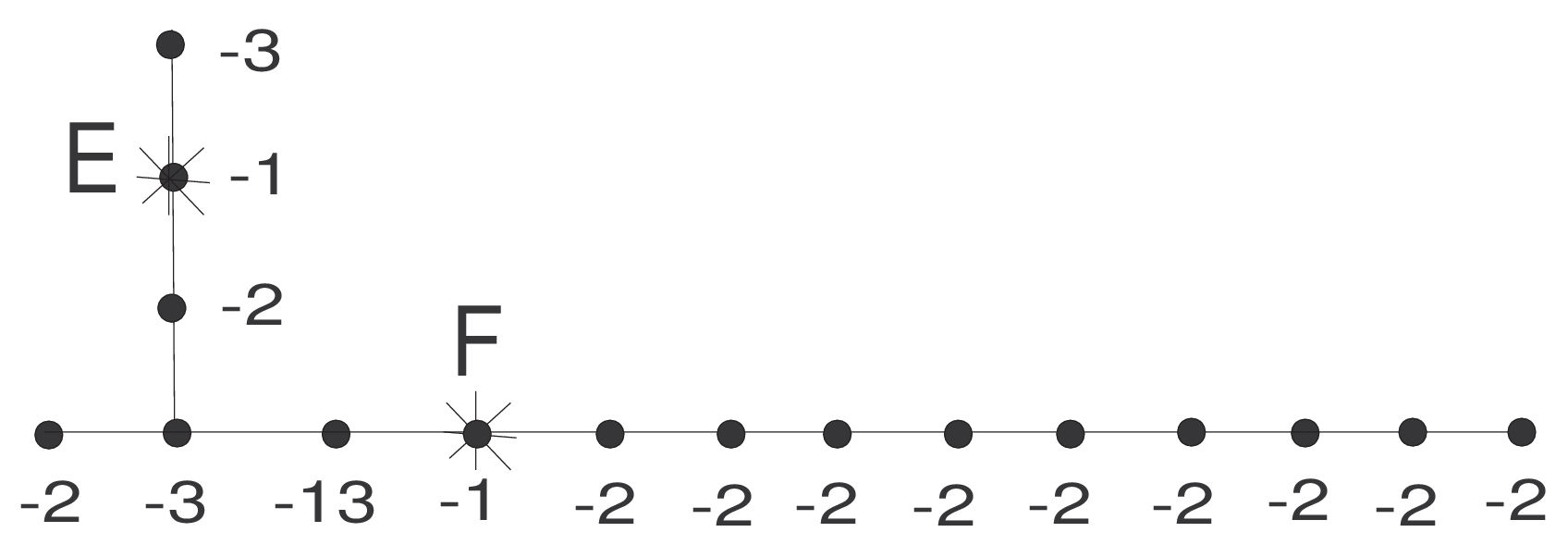}
\caption{Dual graph of the exceptional set of the minimal model of $E+F$ in our example.} 
\label{fig:ex}%
\end{figure}

It is easy to check using Proposition \ref{prop:val} (4) that $\nu_E\leq \nu_F$. 
Morever we have that $a_F(\CC^2)=21$ and $a_E(\CC^2)=17$.

We will obtain $\overline{N}_F\ncontneq \overline{N}_E$ as a consequence of Proposition 
\ref{prop:no_adj}. To prove it we will need Lemma \ref{lem:puis} below, based on the following observation: 

\begin{fact} 
Given a satellite divisor $E$, all the curves $V(h)$ associated with $E$ 
have the same Newton-Puiseux expansion with exactly the same coefficients up to the 
coefficient corresponding to the 
greatest exponent which varies as the strict transform moves along $E$. 

\end{fact}

\begin{lem}\label{lem:puis}
     Let $\alpha(t,s)=(\alpha_1(s,t),\alpha_2(s,t))$ be a wedge realising 
     a Nash-adjacency $\overline{N}_F\contneq\overline{N}_E$ whose generic arc lifts transversely by $E$.     
     Then there exists $t(u,s)\in \CC((s))[[u]]$ such that: 
       \begin{equation}\label{eq:puis}
             \alpha(t(u,s),s)=
             (u^{\beta_0},\sum_{i=\beta_1}^{\beta_k-1}c_iu^i+\lambda(s)u^{\beta_k}+
            \sum_{i>\beta_k}c_i(s)u^i)
       \end{equation}
     where $c_i\in\CC$ for $i=\beta_1,\beta_1+1,...,\beta_{k-1}$ and $\lambda(s)$ and $c_i(s)$ for $i>\beta_k$ 
     are meromorphic in $s$  and such that (\ref{eq:puis}) is the Newton-Puiseux 
expansion for any generic arc $\alpha_{s \neq 0}$. 
\end{lem}

\begin{proof} 
We just perform in family the usual procedure to get the Newton-Puiseux expansion. 
Note that the tangent line of the generic arcs is constant unless the arcs are smooth and 
$E=E_{O}$.  This case is easy so we assume that the tangent of the generic arc is constantly $y=0$.
 
Let $\alpha_1(s,t)$ be equal to $\sum_j a_j(s)t^j$ for some holomorphic germs $a_j(s)$. If $s^{\beta_0}$ 
does not divide $a_{\beta_0}(s)$, we do the substitution $s=s^{\beta_0}$. Then, we can 
write $\alpha_1$ as follows:
      $$\alpha_1(s,t)=s^{\beta_0R}v(s)t^{\beta_0}
           \left(1+\frac{a_{\beta_0+1}(s)}{a_{\beta_0}(s)}t^1+ \cdots \right)$$
with $R>0$,  $v(0)\neq 0$.
Now we can take the $\beta_0$-root of $\alpha_1(t,s)$ for $s\neq 0$ which will be a function: 
$$U(t,s)=ts^R\sqrt[\beta_0]{v(s)}\left(1+\sum_{i>1} b_i(s)t^i  \right)$$ 
for some $b_i(s)$ possibly meromorphic in $s$. Observe that since $v(s)$ is holomorphic, 
so is  $\sqrt[\beta_0]{v(s)}$.

Then, we have that: 
    \begin{equation}\label{eq:U}
           U(t,s)^{\beta_0}=\alpha_1(t,s).
    \end{equation} 
Consider the inverse function $T(u,s)$ of $U(t,s)$ in $\CC((s))[[t]]$, that is, the unique element 
in $\CC((s))[[t]]$ such that $U(T(u,s),s)=u$. 
Substituting $t=T(u,s)$ in $\ref{eq:U}$, we have that $\alpha_1(T(u,s),s)=u^{\beta_0}$ 
for all $s\neq 0$. Then, $\alpha(T(u,s),s)$ is in the Newton-Puiseux form for $s\neq 0$, 
so it is constant in $s$ up to order $\beta_k-1$ by the previous fact.
\end{proof}

Making the change of variables $u=u(t)$ (where $u(t)$ is the inverse function of $t(u)$) 
we get the equivalent formulation:
\begin{cor}\label{cor:puis}
       Let $E$ be any prime divisor and let $(u^{\beta_0}, \sum_{i\geq\beta_1}^{\beta_k}c_iu^i+
         \lambda u^{\beta_k})$, with $c_i$, $\lambda\in \CC$, be the Newton-Puiseux 
         expansion of an associated curve $V(h_E)$. Any wedge $\alpha$ with generic arc in 
         $\dot{N}_E$  is obtained 
         from an expression of the form:  
         $$\left(u^{\beta_0}, \sum_{i\geq\beta_1}^{\beta_k-1}c_iu^i+
              \lambda(s)u^{\beta_k}+\sum c_i(s)u^i \right)$$ 
         for some $\lambda(s)$ and $c_i(s)$ for $i>\beta_k$ meromorphic in $s$, by performing 
         a meromorphic change of coordinates $u=u(t)$ with:  
         $$u(t)=s^Rtw(s) \left(1+\sum_{i\geq 1} b_i(s)t^i \right)$$ 
         for some $R>0$, some  holomorphic function $w(s)$ and some meromorphic 
         functions $b_i(s)$. 
\end{cor}

As a consequence:

\begin{prop}\label{prop:no_adj} 
  Let $E$ be a prime divisor such that the characteristic sequence of its associated curves is 
  $(\beta_0, \beta_1, ..., \beta_k)$ (see Subsection \ref{sec:speprime}). 
  Let $F$ be a prime divisor with contact order $1$ with $E$.
  If $m_O(h_F)=\beta_1+1$ for an associated curve $V(h_F)$, 
  then $\overline{N}_F\nsubseteq \overline{N}_E$. 
\end{prop}

\begin{proof} 
Assume we have $\overline{N}_F\contneq\overline{N}_E$. Using Theorem \ref{theo:Nash_top}, 
we know that $\overline{N}_{F'}\contneq\overline{N}_{E'}$ for any $(E',F')\equiv (E,F)$. 
Observe that the characteristic exponents and the multiplicities of curves associated to $E$ 
and $F$ do not change if we replace $(E,F)$ by a
combinatorially equivalent pair $(E',F')$. Therefore, in order to prove the statement we can assume that the 
tangent lines to the curves associated with $F$ and $E$ are respectively $x=0$ and $y=0$ and 
that the curves associated with $E$ have Newton-Puiseux expansion:  
  $$(u^{\beta_0}, u^{\beta_1}+ \cdots  +u^{\beta_{k-1}}+\lambda u^{\beta_k})$$ for some 
$\lambda\in \CC$. Then, if $\overline{N}_F\contneq\overline{N}_{E'}$, then there exists a (holomorphic) wedge 
$\alpha(s,t)$ realising this Nash-adjacency. By the previous corollary, 
the wedge $\alpha$ can be obtained by a change of variables $u=u(t,s)$ with  
$u(t,s)=ts^Rv(s)(1+\sum c_i(s)t^i)$ for some 
$R>0$, $v(s)\in \CC[[s]]$ with $v(0)\neq 0$ and $c_i(s)\in \CC((s))$, in:  
$$\beta(u,s)=(\beta_1(u,s),\beta_2(u,s))=(u^{\beta_0}, u^{\beta_1}+ \cdots +u^{\beta_{k-1}}+
    \lambda(s) u^{\beta_k}+ \cdots )$$  for some $\lambda(s)\in \CC((s))$. 

In particular, this means that $\alpha(t,s)=(\alpha_1(t,s),\alpha_2(t,s)):=\beta(u(t,s),s)$ is 
a holomorphic function and that $\alpha_2(t,0)$ equals $\mu t^{m(E)}+ \cdots$ 
for some $\mu\in \CC^*$.

Now, on the one hand, the coefficient of $t^{\beta_1+1}$ in $\alpha_2(t,s)$ is: 
$$\beta_1s^{R\beta_1}v(s)^{\beta_1}c_1(s).$$ 
We know that this expression does not have either poles or zeros at $s=0$ (recall that $m_O(h_F)$ 
equals $\beta_1+1$). This only happens when $c_1(s)$ 
has exactly a pole of order $R\beta_1$ at $s=0$. 

On the other hand, the coefficient of $t^{\beta_0+1}$ in $\alpha_1(t,s)$ is 
$\beta_0s^{R\beta_0}v(s)^{\beta_0}c_1(s)$. We want it to be holomorphic, so $c_1$ 
cannot have a pole at $s=0$ of order greater that $R\beta_0$. But this is not compatible 
with the previous requirement. We can conclude that such a wedge does not exist.
\end{proof}



\end{document}